\documentclass[12pt,a4paper]{amsart}
\usepackage[T1]{fontenc}
\usepackage{amsmath}
\usepackage{amssymb}
\usepackage{graphicx}
\usepackage{amsfonts}
\usepackage{amsthm}
\usepackage{amsmath}
\usepackage{amscd}
\usepackage[latin2]{inputenc}
\usepackage{t1enc}
\usepackage[mathscr]{eucal}
\usepackage{indentfirst}
\usepackage{graphicx}
\usepackage{graphics}
\usepackage{pict2e}
\usepackage{epic}
\numberwithin{equation}{section}
\usepackage[margin=2.9cm]{geometry}
\usepackage{epstopdf}

\theoremstyle{plain}
\newtheorem{theorem}{Theorem}[section]
\newtheorem{lemma}[theorem]{Lemma}
\newtheorem{corollary}[theorem]{Corollary}
\newtheorem{proposition}[theorem]{Proposition}
\newtheorem{assumption}{Assumption}

 \theoremstyle{definition}

\newtheorem{?}[theorem]{Problem}

\newcommand{\diam}{{\rm{diam}}}

\newcommand{\argmax}{\text{argmax}}

\begin{document}

\title{Can we trust Bayesian uncertainty quantification from Gaussian process priors with squared exponential covariance kernel?}

\author[A. Hadji]{Amine Hadji}

\address{Leiden University \\ Mathematical Institute \\
\\ The Netherlands} 

\email{m.a.hadji@math.leidenuniv.nl}

\author[B. Sz\'{a}bo]{Botond Sz\'{a}bo}

\address{Leiden University \\ Mathematical Institute \\
\\ The Netherlands} 

\email{b.t.szabo@math.leidenuniv.nl}

 \subjclass[2019]{68Q25, 68R10, 68U05}

 \keywords{credible set, frequentist coverage, empirical Bayes, adaptation, asymptotics}

\begin{abstract}
We investigate the frequentist coverage properties of credible sets resulting in from Gaussian process priors with squared exponential covariance kernel. First we show that by selecting the scaling hyper-parameter using the maximum marginal likelihood estimator in the (slightly modified) squared exponential covariance kernel the corresponding credible sets will provide overconfident, misleading uncertainty statements for a large, representative subclass of the functional parameters in context of the Gaussian white noise model. Then we show that by either blowing up the credible sets with a logarithmic factor or modifying the maximum marginal likelihood estimator with a logarithmic term one can get reliable uncertainty statement and adaptive size of the credible sets under some additional restriction. Finally we demonstrate on a numerical study that the derived negative and positive results extend beyond the Gaussian white noise model to the nonparametric regression and classification models for small sample sizes as well.
\end{abstract}

\maketitle

\section{Introduction}

Bayesian methods are frequently applied in various fields of applications. One, very appealing advantage of the Bayesian framework is that it readily provides built-in uncertainty quantification. In nonparametric problems uncertainty statements are visualized with credible bands, i.e. bands accumulating prescribed fraction (typically $95\%$) of the posterior mass. Gaussian processes are popular and frequently used choices for prior distributions in high-, and infinite dimensional models. Areas of possible applications include machine learning \cite{rasmussen:williams:2006}, astronomy \cite{fmswwz:astonomy:16}, genomics \cite{kalaitzis:lawrence:2011}, linguistics \cite{koriyama:kobayashi:2015}, epidemiology \cite{bhatt:2015},...etc.

Gaussian processes are characterized by their mean and covariance kernel. Typically one considers centered Gaussian priors. For covariance kernel arguably one of the most frequently used choice is the squared exponential kernel $K(\cdot,\cdot)$, i.e. for centered Gaussian process $G_t$,
\begin{align}
K(s,t)=E[G_sG_t] = b\exp\{-a(t-s)^2\},\qquad s,t\in T,\label{def: kernel}
\end{align}
for given hyper-parameters $a,b>0$, see for instance  \cite{rasmussen:williams:2006}. In our paper, we focus on the effect of the hyper-parameter $a$, which has been investigated in the context of various models including nonparametric regression, density estimation, and classification, see for instance  \cite{vzanten:vdv:07,castillo:2014}. We take the hyper-parameter $b=1$ fixed for simplicity. 

In our work we adopt a frequentist perspective, i.e. we assume that the data is generated from some unknown, but fixed probability distribution $P_{f_0}$, indexed by a true underlying functional parameter of interest $f_0$. We are interested in recovering $f_0$ using Bayesian methodology or in other words we are interested in the performance of Bayesian techniques for recovering the underlying functional parameter $f_0$.  We consider the asymptotic regime assuming that sample size or signal-to-noise ratio increases indefinitely. The frequentist properties of Bayesian methods has been extensively studied in the literature in general high-dimensional and non-parametric settings, see for instance \cite{ghosal:ghosh:vdv:00,ghosal:vdv:07,rousseau:szabo:2015:main}. In these papers it was shown that under relatively mild conditions on the prior and the likelihood function the posterior distribution contracts around the true parameter of interest at (in many instances) optimal rate in various high-dimensional and non-parametric problems. These results show that not only the point estimators (e.g. posterior mean, $\delta$-posterior mode, etc...) resulting from the posterior provide typically good recovery of the truth under the frequentist data generating process, but also the spread of the posterior is not too large and most of the posterior mass is concentrated in a ball centered around the truth with optimal diameter.

In the literature, due to its importance, the posterior contraction rates associated to  Gaussian process priors with squared exponential kernel were also separately investigated and it was shown that for appropriate choice of the hyper-parameter $a$, depending on the smoothness of the underlying functional parameter of interest, the posterior distribution achieves nearly the optimal (minimax) contraction rate, see for instance \cite{vzanten:vdv:07,vandervaart2009,castillo:2014,bhattacharya:pati}. In practice the regularity of the underlying function is typically not known, therefore one either endows the scaling hyper-parameters with an additional layer of prior distribution (resulting in the so called hierarchical prior distribution, see also \cite{vandervaart2009}) or estimate them from the data (using typically the maximum marginal likelihood estimator). In our work we focus on investigating the reliability of these procedures for uncertainty quantification. Assuming that our observations are generated via a true $f_0$,  we are interested in whether the credible sets contain this function. In our theoretical analysis we consider the Gaussian white noise model, which is closely related to various nonparametric models and can be thought of as the idealized, continuous observation version of nonparametric regression model.

Although Bayesian methods are routinely used for uncertainty quantification, it is rather unclear whether they can provide reliable uncertainty statements. In fact it is well known that it is impossible to construct confidence sets (either based on purely frequentist or Bayesian methods) which has optimal size over a wide range of regularity classes and provides reliable uncertainty quantification simultaneously. More precisely let us assume that we have a collection of functional classes $\Theta^{\beta}$ indexed by some regularity hyper-parameter $\beta\in B$ and let us denote the minimax estimation rate corresponding to this class by $r_{n,\beta}$ (where $n$ denotes the signal-to-noise ratio). Then in general it is not possible to construct an ``honest'' confidence set $\hat{C}_n$  which achieves simultaneously that
\begin{align*}
\liminf_n\inf_{\beta\in B}\inf_{f\in\Theta^{\beta}} P_{f}(f\in \hat{C}_n)\geq 1-\alpha,\\
 \liminf_n\inf_{\beta\in B}\inf_{f\in\Theta^{\beta}} P_{f}(\|f\| \leq \hat{C} r_{\beta})\geq 1-\alpha,
\end{align*}
for some given significance level $\alpha>0$, see for instance \cite{Low:97,cai:low:04,robins:2006,gine:nickl:2016}. Therefore one has to introduce additional assumptions on the functional parameter $f_0$ to obtain confidence sets with optimal size and reliable uncertainty quantification. In the literature various constraints were proposed to overcome this problem, see for instance the monograph \cite{gine:nickl:2016} for a collection of such approaches.

The coverage properties of credible sets have been investigated only recently,  see for instance in \cite{castillo:nickl:2013,szabo:vdv:vzanten:13,szabo:etal:2015,rousseau:szabo:16:main,pas:etal::2016, belitser:2014,sniekers:2015,ray2017,castillo:szabo:2018}, and references therein, for various combinations of models and priors.  In these papers it was shown that for appropriate choices of the prior distribution both the hierarchical and empirical Bayes procedures can provide in various nonparametric models reliable uncertainty statements under some additional regularity assumption on the underlying functional parameter (e.g. self-similarity assumption, polished tail condition, excessive bias assumption, etc). In our work we focus on the Gaussian process prior with (approximately) squared exponential kernel and show that the empirical Bayes procedure results in unreliable uncertainty statement for a large class of functions. The derived negative theoretical results are also demonstrated via simulation study, and extended to the nonparametric regression and classification models. This troubling finding might shatter our trust in this popular and frequently applied Bayesian technique. However, we propose a simple and intuitive fix for this problem by slightly modifying the maximum likelihood estimator used in the empirical Bayes method. This modification corrects for the haphazard, unreliable uncertainty statements in theory and we demonstrate its applicability in a numerical example for small sample sizes as well.

The paper is organized as follows. In Section \ref{sec: model} we introduce the Gaussian white noise model and the considered Bayesian approach in details. Then we present first in Section \ref{sec: uncertainty} our negative findings on the coverage properties of Bayesian $L_2$-credible sets, and then propose different modifications correcting the haphazard behaviour of the posterior by either blowing up the credible sets by a logarithmic factor or (slightly) modifying the marginal maximum likelihood estimator. We show that the proposed methods indeed correct the overconfident uncertainty statements and result in reliable uncertainty quantification for polished tail and self-similar functions, respectively.  We demonstrate our findings on a simulation study in Section  \ref{sec: simulation} and discuss the derived results and possible extensions in Section \ref{sec: discussion}. The proofs of the above results are deferred to the Appendix. Finally as a byproduct we also derive contraction rates for the empirical and hierarchical Bayes procedures for a wide range of priors on the rescaling hyper-parameter, extending the results available in the literature. The contraction rate results and their proofs are also deferred to the Appendix.

\section{Main results}
\subsection{Model description}\label{sec: model}~
We consider the Gaussian white noise model
\begin{align}\label{eq: model}
Y(t)=\int_0^{t}f_0(s)ds+\frac{1}{\sqrt{n}}W_t, \quad t\in[0,1],
\end{align}
where $f_0\in L_2[0,1]$ is the unknown function of interest and $W_t$ denotes the Brownian motion. Let $P_0,E_0$, and $V_0$ denote the corresponding probability measure, expected value, and variance, respectively. This model is closely related to the popular nonparametric regression and density estimation models \cite{nussbaum:1996,brown:1996} and can be used as a platform to investigate more complex statistical models, see for instance \cite{tsybakov:2009,gine:nickl:2016}. In the Bayesian approach we endow the unknown function of interest $f_0$ with a prior distribution representing our initial belief. In our work we investigate the popular Gaussian process prior with rescaled squared exponential kernel \eqref{def: kernel}. Let us consider the sequence representation of the Gaussian white noise model. For an orthonormal basis $\psi_i$, $i=1,2,...$ (e.g. the Fourier basis) let us denote the sequence decomposition of the functions $f_0(t), Y(t)$, and $W_t$ by $Y_i=\langle Y(t),\psi_i(t) \rangle_2$, $f_{0,i}=\langle f_0,\psi_i(t) \rangle_2$, and $Z_i=\langle W_t,\psi_i(t) \rangle_2\stackrel{iid}{\sim}N(0,1)$, $i=1,2,...,$ respectively. Then the equivalent sequence model can be given in the form
$$ Y_i=f_{0,i}+\frac{1}{\sqrt{n}}Z_i, \quad i=1,2,... .$$
Slightly abusing our notations we denote by $f_0$ both the functional parameter in the Gaussian white noise model and the sequential parameter $f_0=(f_{0,1},f_{0,2},...)$ in the sequence model. It is common to assume that the true function $f_0$ belongs to a hyper-rectangle regularity class, i.e.
$$ f_0\in \Theta^{\beta}(M)=\{f\in \ell_2 : \sup_{i\geq 1} f_{i}^2i^{2\beta+1}\leq M\}, $$
for some (typically unknown) $\beta,M>0$. We note that the minimax estimation rate for the above hyper-rectangle regularity class is $n^{-\beta/(1+2\beta)}$, i.e. there exists $C_{\beta}>0$ such that
\begin{align*}
\inf_{\hat{f}}\sup_{f\in\Theta^{\beta}(M)}E_0 \|f-\hat{f}\|_2\geq  C_\beta n^{-\beta/(1+2\beta)},
\end{align*}
where the infimum is taken over all estimators, see for instance \cite{donoho1990}.

In view of Mercel's theorem we can represent the Gaussian process prior with squared exponential kernel as
$$ G_t=\sum_{i=1}^{\infty} \lambda_i \xi_i \psi_i(t), $$
where $\lambda_i$, $\psi_i$, $i=1,2,...$ are the eigenvalues and eigenfunctions of the squared exponential covariance kernel, and $\xi_i$ are iid standard normal random variables, see for instance Chapter 4.3 of \cite{rasmussen2004gaussian}. The corresponding coefficients $\lambda_i$ can be approximated as $\lambda_i^2\approx a^{-1}e^{-i/a}$, see for instance \cite{rasmussen2004gaussian,bhattacharya:pati}. In the rest of the paper we work with the prior 
\begin{align}
f|a\sim\bigotimes_{i=1}^\infty N(0,a^{-1}e^{-i/a})\label{def: prior}
\end{align}
in the sequence model, for convenience. Note that in view of $Y|f_0\sim\bigotimes_{i=1}^\infty N(f_{0,i},n^{-1})$ and the choice of the prior $\Pi_a(\cdot)$ in \eqref{def: prior}, the posterior $\Pi_a(\cdot|Y)$ takes the form 
\begin{align}\label{eq: posterior}
f|a,Y \sim \bigotimes_{i=1}^\infty \mathcal{N}\Bigg(\frac{nY_i}{ae^{i/a}+n},\frac{1}{ae^{i/a}+n}\Bigg).
\end{align}

\subsection{Uncertainty quantification}\label{sec: uncertainty}
In our work we investigate the reliability of the built-in uncertainty quantification of the above posterior distribution. In Bayesian methods the remaining uncertainty of the procedure is visualized by the credible set. We consider $\ell_2$-credible balls, i.e. we analyze credible sets in the form
\begin{align}
\hat{C}_{n,\alpha}=\{f\in \ell_2:\, \|f-\hat{f}_a\|_2\leq r_{\alpha}(a)\}\label{def: credible_a}
\end{align}
where $\hat{f}_a$ is the posterior mean and the radius $r_{\alpha}(a)$ is chosen such that $\Pi_a(f\in \hat{C}_{n,\alpha}|Y)=1-\alpha$, for some prescribed significance level $\alpha>0$ and hyper-parameter $a>0$.

The behaviour of the posterior distribution substantially depends on the choice of the hyper-parameter $a$. For $a= n^{1/(1+2\beta)}$ the corresponding posterior achieves reliable uncertainty quantification in the sense that the credible sets are asymptotically confidence sets, see for instance \cite{vzanten:vdv:07,bhattacharya:pati}.  In practice, however, the true smoothness parameter $\beta$ is unknown, hence one has to use the data to find the optimal scaling hyper-parameter $a$. The two most commonly applied Bayesian methods for selecting the hyper-parameter are the hierarchical Bayes and the marginal likelihood empirical Bayes methods. In the hierarchical Bayes method the hyper-parameter $a$ is endowed with a prior distribution $\pi$ (also called hyper-prior distribution), resulting in a two-level, hierarchical prior distribution
\begin{align*}
\Pi(\cdot)=\int_0^{\infty} \Pi_a(\cdot)\pi(a)da.
\end{align*}
 In contrast to this in the empirical Bayes approach we take the maximum marginal likelihood estimator (MMLE), i.e.
\begin{align}
 \hat{a}_n := \arg\max\limits_{a \in [1,A_n]}\ell_n(a), \label{def: MMLE}
\end{align}
where the marginal log-likelihood function (with respect to the measure $\bigotimes_{i=1}^\infty N(0,1)$) is
 $$\ell_n(a) = -\frac{1}{2}\sum\limits_{i=1}^{\infty}\Bigg( \log\Big(1 +\frac{n}{ae^{i/a}}\Big) - \frac{n^2Y_i^2}{ae^{i/a}+n}\Bigg)$$
and the parameter $A_n=o(n)$ restricts the parameter space to a compact interval, which is advantageous both from practical and analytical perspective. Then the estimator $\hat{a}_n$ is plugged in into the posterior distribution \eqref{eq: posterior}.


For technical reasons, we introduce the following assumptions on the hyper-prior density function $\pi(\cdot)$ supported on $[1,A_n]$.

\begin{assumption}\label{assump: HB}
Let us assume that for some $c_1>0$ there exist $c_2,c_6\geq 0$ and $c_3, c_4,c_5>0$ such that
\begin{align}
c_4^{-1}a^{-c_3}\exp(-c_2a)\leq\pi(a)\leq c_4 a^{-c_5}\exp(-c_6a),
\end{align}
for all $c_1\leq a\leq A_n$.
\end{assumption}
Note that amongst others the exponential, the gamma, and the inverse gamma distributions (restricted to $[1,A_n]$) satisfy Assumption \ref{assump: HB}. 

 We are interested in the frequentist properties of $\ell_2$-credible balls resulting from the data driven credible balls. For convenience let $\Pi_n(\cdot| Y)$ denote both the hierarchical and the empirical Bayes posterior distributions in the following. Then let us denote by $r_{\alpha}$ the radius of the $\ell_2$-ball centered around the posterior mean $\hat{f}$ and accumulating $1-\alpha$ fraction of the posterior mass, i.e.
$$\Pi_n(f:\,  \|f-\hat{f}\|_2\leq r_{\alpha}|Y )=1-\alpha.$$
In our analysis we introduce some additional flexibility by considering inflated credible balls, i.e.
\begin{align}
 \hat{C}_n(L_n)=\{f:\,  \|f-\hat{f}\|_2\leq L_n r_{\alpha} \},\label{def: credible}
\end{align}
for some blown up factor $L_n\geq 1$, possibly depending on $n$.
As a first step we note that the size of the credible set for both the empirical and hierarchical Bayes procedures adapts to the minimax rate (actually the diameter of the set is even a logarithmic factor faster than the minimax rate).

\begin{corollary}\label{cor: size}
Both the hierarchical and the empirical Bayes credible sets defined in \eqref{def: credible} have rate adaptive size, i.e. for every $\beta_0>0$ and $M>0$
\begin{align*}
\sup_{\beta\geq \beta_0}\sup\limits_{f\in\Theta^{\beta}(M)} P_{\theta_0}\Big(\diam( \hat{C}_n(1) )\geq M_n n^{-\beta/(1+2\beta)}(\log n)^{-1/(1+2\beta)}\Big)\to 0,
\end{align*}
where the sequence $M_n$ goes to infinity arbitrary slowly in case of the empirical Bayes method and $M_n\gg \log n$ in case of the hierarchical Bayes method, and $\diam(S)$ denotes the $\ell_2$-diameter of the set $S\subset\ell_2(M)$.
\end{corollary}
\begin{proof}
The proof is given in Sections \ref{sec: self similar c-e} and \ref{sec: hb}.
\end{proof}
\subsubsection{Coverage of credible sets - negative results}
Next we investigate how much we can trust the above derived data-driven Bayesian uncertainty quantification from a frequentist perspective.  We would like to know whether the true function $f_0$ is included in the (blown up) credible set, i.e. if
$$ \inf\limits_{f_0\in \cup_{\beta\geq \beta_0} \Theta^{\beta}(M) }P_{0}(f_0\in\hat{C}_n(L_n))\geq 1-\alpha $$
holds for some sufficiently large choice of $L_n$? Since it is impossible to construct honest confidence sets with rate adaptive size and in view of the adaptive size of the credible sets (see Corollary \ref{cor: size}), they must have poor frequentist coverage properties at least for certain functional parameters $f_0$. Actually the radius of the credible sets are even faster than the minimax rate, which already implies impossibility of coverage. Nevertheless it is of interest to quantify the set of functions for which the Bayesian uncertainty quantification is truth-worthy. 

First we note that a representative subset of the hyper-rectangle $\Theta^\beta(M)$ is the set
\begin{align}
\Theta^{\beta}_{s}(m,M)=\{f\in \Theta^{\beta}(M):\, \min_{i\geq 1} i^{1+2\beta}f_{0,i}^2\geq m\},\label{def: selfsim}
\end{align}
for some parameters $0<m\leq M$. Let us refer to this subclass of sequential parameters as self-similar signals following the similar terminology of \cite{gine:2010, szabo:vdv:vzanten:13}. It was shown in the later paper that the minimax rate over $\Theta^{\beta}_{s}(m,M)$ is the same as over $\Theta^\beta(M)$. The next theorem shows that the empirical Bayes procedure provides unreliable uncertainty quantification over this representative sub-class of functions unless it is blown up with at least a logarithmic factor.

\begin{theorem}\label{thm: self similar c-e}
Let us take arbitrary $L_n=o(\sqrt{\log n})$. Then the empirical Bayes credible set blown up by $L_n$ has frequentist coverage tending to zero for every self-similar signal, i.e. for every $0<m\leq M$
\begin{align*}
\sup_{f_0\in \Theta^{\beta}_{s}(m,M)} P_{0}(f_0\in\hat{C}_n(L_n))\to 0.
\end{align*}
\end{theorem}
\begin{proof}
See Section \ref{sec: self similar c-e}.
\end{proof}

This negative result draws a dark picture as it tells us that one can not trust Bayesian uncertainty quantification resulting from the investigated prior, even if one allows certain amount of adjustment (i.e. by blowing up the set with a sequence tending to infinity, not too fast). Since the investigated prior is very closely related to the Gaussian process with squared exponential covariance kernel this gives the intuition that one has to be very cautious working with squared exponential kernel as the corresponding Bayesian uncertainty statement are (typically) unreliable. In the next subsection we will be touching the corners by deriving some positive results on the coverage properties of the credible sets. First we show that for analytic functions the (slightly inflated) credible sets provide reliable uncertainty quantification and second we show that by blowing up the credible sets by a $\log n$ factor or by slightly adjusting the maximum marginal likelihood estimator, one gets reliable uncertainty statements for a large subclass of functions, including the self-similar functions.

\subsubsection{Coverage of credible sets - positive results}
Let us consider the set of analytic-type functions defined as
$$ f_0\in A^{\gamma}(M)=\{f\in  \ell_2(M): \sum_{i=1}^{\infty}f_{i}^2e^{2i\gamma}\leq M\}, $$
for some $\gamma>0$. Note that the investigated prior \eqref{def: prior} is more suitable for this class of functions  due to the exponential decay of the variances. We show below that, indeed, for the class $A^{\gamma}(M)$ the empirical Bayes procedure provides reliable uncertainty quantification. Note, however, that the present class of functions is substantially smaller than $\Theta^{\beta}(M)$, for any $\beta>0$.

\begin{theorem}\label{thm: analytic coverage}
The inflated empirical Bayes credible set $\hat{C}_n(L)$ has frequentist coverage tending to one over the class $f_0\in A^{\gamma}(M)$ for any $\gamma\geq 1/2$ and sufficiently large constant $L>0$, i.e. 
$$ \inf\limits_{f_0\in A^{\gamma}(M)}P_{0}(f_0\in\hat{C}_n(L))\to 1. $$
Furthermore, the size of the credible set is (nearly) optimal, i.e. for some sufficiently large constant $C>0$,
$$\inf\limits_{f_0\in A^{\gamma}(M)} P_0 \big(\diam(\hat{C}_n(1))\leq C n^{-1/2}\log n  \big)\to 1.$$
\end{theorem}
\begin{proof}
See Section \ref{sec: analytic coverage}.
\end{proof}

Next we investigate the behaviour of the credible sets by allowing a logarithmic inflating factor. Since the size of the inflated credible sets are still nearly minimax, the credible sets fail to cover all functional parameter $f_0$ of interest, in view of the non-existence result of adaptive confidence sets \cite{cai:low:04,robins:2006}. Therefore we restrict the investigated class of functions to the so called polished tail class, introduced in \cite{szabo:vdv:vzanten:13,rousseau:szabo:16:main}. We say that a sequential parameter $f\in\ell_2(M)$ belongs to the class of polished tail signals denoted by $\Theta_{pt}(L_0,N_0,\rho)$, for some $L_0,\rho, N_0>0$ if
$$ \sum\limits_{i=N}^{\infty}f^2_i \leq L_0\sum\limits_{i=N}^{\rho N}f^2_i, \qquad\text{for all $N\geq N_0$}. $$
The above assumption basically requires that knowing the sequential parameter $f$ up to a certain coordinate enables us to draw conclusion about the tail of the sequence. We require that the energy (sum of the squared coefficients) of the tail is dominated by the energy of a finitely large block of coefficients. This condition makes also sense intuitively as in the stochastic model the signal can be observed only up to some limit, the fluctuation in the later coordinates can equally likely be caused by the noise. Therefore to make reliable uncertainty statement we have to assume that the tail behaviour of the signal hidden by the noise is not substantial and can be extrapolated by information available at given signal-to-noise ratio. In \cite{szabo:vdv:vzanten:13} it was shown that the above assumption is mild from statistical,  topological and Bayesian point of view.

The next theorem states that when the sequential parameter $f_0$ is restricted to polished tail sequences, then both the empirical and hierarchical Bayes credible balls blown up by a $\log n$ factor (i.e $\hat{C}_n(L\log n)$) are honest frequentist confidence set, if $L$ is large enough.
\begin{theorem}\label{thm: polish tail coverage}
For any $L_0,N_0,\rho\geq 1$ there exists a constant $L$ such that
$$ \inf\limits_{f_0\in\Theta_{pt}(L_0,N_0,\rho)}P_{0}(f_0\in\hat{C}_n(L\log n))\to 1, $$
where $\hat{C_n}$ denotes either the empirical or hierarchical Bayes credible sets under Assumption \ref{assump: HB}.
\end{theorem}
\begin{proof}
See Sections \ref{sec: polish tail coverage} and \ref{sec: hb:Thm2.4}.
\end{proof}

Hence one can achieve reliable uncertainty quantification on an arguably large subset of the parameter space by blowing up the standard credible set with a slowly varying term. This, however, is not very appealing as a practitioner would righteously hesitate introducing the artificial logarithmic blow up. Therefore, we propose another method, where one does not have to introduce a logarithmic blow up factor, but instead adjust the maximum marginal likelihood estimator. Investigating the proof of Theorem \ref{thm: self similar c-e} one can see that the MMLE $\hat{a}_n$, given in \eqref{def: MMLE}, is too small, the empirical Bayes procedure is oversmoothing. One can compensate for this by undersmoothing the procedure. We propose to adjust the MMLE by a multiplicative logarithmic factor
\begin{align}
\tilde{a}_n=\log (n)\hat{a}_n.\label{def: mod_mmle}
\end{align}
Then the corresponding empirical Bayes credible set (blown up by a sufficiently large constant $L>0$) results in reliable uncertainty quantification for self-similar functions $\Theta^{\beta}_s(m,M)$.

\begin{theorem}\label{thm: mod_eb}
For any $0<m\leq M$ there exists a constant $L>0$ such that
$$ \inf\limits_{f_0\in\Theta_{s}^\beta(m,M)}P_{0}(f_0\in\hat{C}_n(L))\to 1, $$
where $\hat{C_n}(1)$ denotes the credible set resulting from the empirical Bayes posterior with hyper-parameter $\tilde{a}_n$.
\end{theorem}
The proof of the theorem is deferred to Section \ref{sec: mod_eb}.

\section{Simulation study}\label{sec: simulation}
In this section we investigate the numerical properties of the  Gaussian process prior with (approximately) squared exponential covariance kernel. First we consider the Gaussian white noise model and the prior \eqref{def: prior}. We show that the corresponding Bayesian uncertainty quantification is misleading for various regularly behaving functional parameters. We also demonstrate that a different choice of the covariance kernel or a modified version of the empirical Bayes procedure results in more accurate uncertainty statements. Then we consider the (from practical point of view) more relevant nonparametric regression and classification models, where we also demonstrate the sub-optimal behaviour of the (standard) empirical Bayes method with squared exponential covariance kernel and show that the proposed modification results in superior performance compared to it.

\subsection{Gaussian white noise model}
First we demonstrate the suboptimal performance of the Gaussian process with (approximately) squared exponential covariance kernel \eqref{def: prior} compared to modified versions of the empirical Bayes procedure and to the Gaussian process prior with polynomially decaying variances in the series representation, see \cite{knapikSVZ2012,szabo:vdv:vzanten:13}. Let us consider the function $f_1\in L_2[0,1]$ given by their Fourrier coefficients $f_{1,i}=i^{-3/2}\sin(i)$, for $i=1,2,...$, respectively, relative to the Fourrier eigenbasis $\psi_i(t)=\sqrt{2}\cos(\pi(i-1/2)t)$. Note that although the function lies outside of the self-similar function class \eqref{def: selfsim}, it has essentially the same behaviour. In Figure \ref{fig1} we visualize the $95\%$ credible sets (light blue or light red), the posterior mean (blue or red) and the true function (black), by simulating 2000 draws from the empirical Bayes posterior distribution and plotting the closest $95\%$ of them in $L_2$-norm to the posterior mean. We note that all credible sets were constructed without any inflation factor, i.e. $L_n=1$ was taken (except of the case where the choice $L_n=\log n$ was pre-specified). The credible sets are drawn for signal-to-noise ratio $n=100$, $500$, $1000$ and $5000$, respectively. We also plot the same credible sets blown-up by a $\log n$ factor, the credible sets obtained by the modified empirical Bayes procedure (where the MMLE of the scaling parameter $a$ was multiplied by $\log n$) and the empirical Bayes credible sets corresponding to the prior $f\sim\otimes_{i=1}^{\infty}N(0,i^{-1-2\alpha})$, with hyper-parameter $\alpha$ estimated by the MMLE. One can see that the standard marginal likelihood empirical Bayes method provides too narrow credible sets failing to cover the underlying true function. Also note that both modifications of the empirical Bayes credible sets provide good coverage, but in contrast to the overly conservative approach of inflating the credible sets with a logarithmic factor the modification of the MMLE results in more informative uncertainty statement (i.e. smaller credible sets).

\begin{figure}[!h]
	\centering
	\includegraphics[scale=0.9]{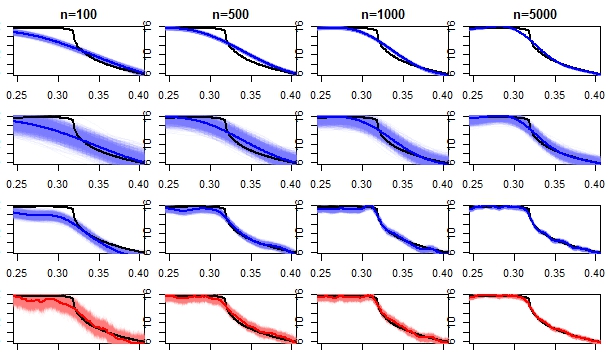}
	\caption{\scriptsize Empirical Bayes credible sets for the function $f_1$ (drawn in black) zoomed in to the interval $x\in[0.25,0.4]$. First line: credible set (in light blue) and posterior mean (blue curve) corresponding to the prior with exponentially decaying variance. Second line: credible set (in light blue) blown-up by a $\log n$ factor ($L_n = \log n$) and posterior mean (blue curve) corresponding to the prior with exponentially decaying variance. Third line: credible set (in light blue) and posterior mean (blue curve) corresponding to the prior with exponentially decaying variance and modified empirical Bayes procedure (rescaling factor multiplied by $\log n$). Last line: credible set (in light red) and posterior mean (red curve) corresponding to the prior with polynomially decaying variance. From left to right the signal to noise ratio is $n=100,500,1000, 5000$.}
\label{fig1}
\end{figure}

\subsection{Nonparametric regression and classification}
In this section we demonstrate on a simulation study that the results derived for the Gaussian white noise model 
generalize to more complicated statistical models as well. We consider the popular nonparametric regression and classification models specifically.  The empirical Bayes posteriors, posterior means and credible sets are computed in both cases using the MatLab package gpml. 

In the nonparametric regression model we assume to observe pairs of random variables $(X_1,Y_1),(X_2,Y_2),...(X_n,Y_n)$, where 
\begin{align*}
Y_i=f_0(X_i)+\varepsilon_i,\qquad \varepsilon_i\stackrel{iid}{\sim} N(0,\sigma^2),\quad X_i\stackrel{iid}{\sim}U(0,1),
\end{align*}
and the aim is to estimate the unknown nonparametric regression function $f_0$. In the Bayesian approach we endow $f_0$ with a Gaussian process prior with squared exponential kernel and estimate the tunning parameter using the MMLE. 

In this simulation study we take the Fourier coefficients of the underlying true function $f_2$ to be $f_{2,i}=i^{-3/2}\cos(i)$, $i=1,2,...$. We take $\sigma^2=1/2$, but in the procedure it is considered to be unknown and estimated with the MMLE $\hat{\sigma}^2$. We plot the true function (black), the posterior mean (blue), and the posterior pointwise credible intervals (dashed blue) $[\hat{f}(x)-q_{0.025}\sqrt{\hat{c}(x,x)},\hat{f}(x)+q_{0.025}\sqrt{\hat{c}(x,x)}]$, where $\hat{f}$ is the posterior mean, $q_\alpha$ the $\alpha$-th quantile of the standard normal distribution and $\hat{c}(.,.)$ the posterior covariance kernel. We consider the MMLE empirical Bayes method with and without the $\log n$ inflation factor for the credible set, the modified empirical Bayes method (where the MMLE was multiplied by $\log n$), and finally the empirical Bayes method for Mat\'ern covariance kernel.  We take the sample size to be $n=100,500,1000$ and $2000$. Observe in Figure \ref{fig: regression} that the standard MMLE empirical Bayes method provides unreliable uncertainty quantification in certain points, while the two modified squared exponential credible sets and the empirical Bayes credible sets for the  Mat\'ern kernel capture the underlying functional parameter of interest better. Also note that by multiplying the MMLE of the scaling parameter by a $\log n$ factor in the squared exponential kernel case we do not get an overly conservative credible set, unlike in the case when the radius is inflated with a logarithmic factor. Finally, we note that the computation time using the squared exponential kernel is much smaller than working with the Mat\'ern kernel.

We also investigate empirically the frequentist coverage probabilities of the pointwise credible sets by repeating the experiment 100 times and reporting the frequency that the function at given points (we consider $x=(0.25,0.3188,0.75)$ with $0.3188=\argmax_{x\in[0,1]} f_2(x)$) is included in the credible interval, see Table \ref{table: regression1}. Moreover, Table \ref{table: radius} shows the average size of the pointwise credible intervals (i.e. $2q_{0.025}\sqrt{\hat{c}(x,x)}$) depending on the sample size $n$ and the procedure used to compute the credible sets. One can observe similar behaviour to what we have described above. We note that we have left out the Mat\'ern kernel from Table \ref{table: radius} due to the overly long computational time.

\begin{figure}[!h]
	\includegraphics[scale=0.5]{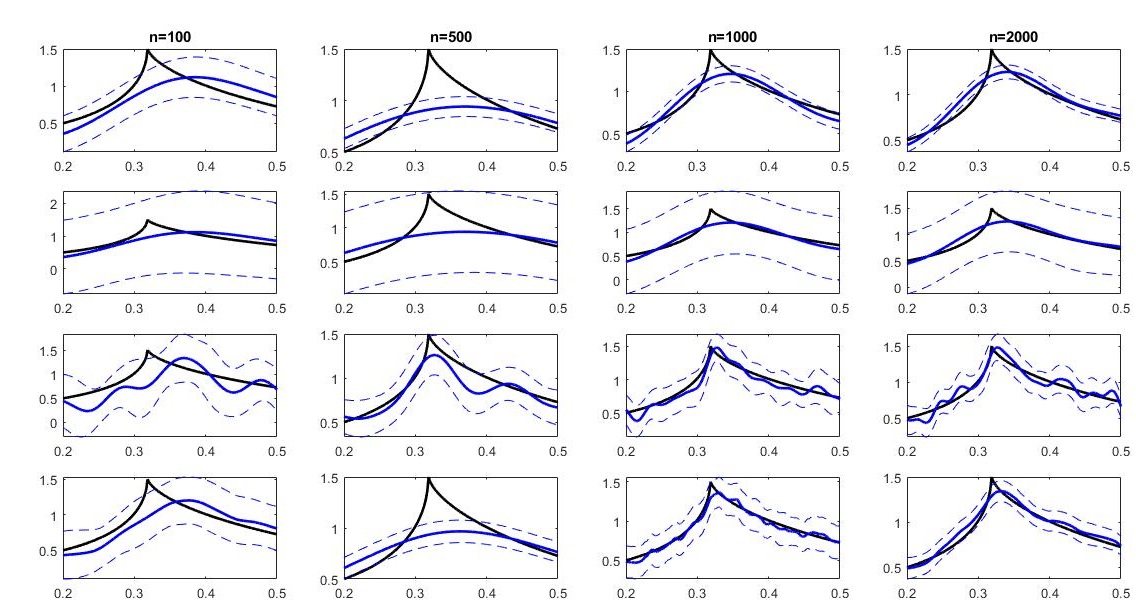}
	\caption{\scriptsize Empirical Bayes credible sets for the  regression function $f_2$ (drawn in black), zoomed in to the interval $x\in[0.2,0.5]$. The posterior means are drawn by solid blue line, while the $95\%$ pointwise credible sets by dashed blue curves. In the first row we plot the MMLE empirical Bayes method, in the second row the MMLE empirical Bayes method with a $\log n$ blow up factor, the third row the modified MMLE empirical Bayes method using squared exponential Gaussian process prior, while in the fourth row we plot the MMLE empirical Byes credible sets using a Matern kernel. From left to right the sample size is $n=100,500,1000,2000$.}
\label{fig: regression}
\end{figure}

\begin{table}[!h]
\centering
	\begin{tabular}{c|c|c|c|c|c|c|c|c|c}
		& \multicolumn{3}{c|}{$x=0.25$} & \multicolumn{3}{c|}{$x=0.3188$} & \multicolumn{3}{c}{$x=0.75$}\\
		$n =$ & $100$ & $500$ & $1000$ & $100$ & $500$ & $1000$ & $100$ & $500$ & $1000$\\ \hline
		Method 1 & 0.84 & 0.69 & 0.57 & 0.01 & 0.01 & 0.00 & 0.98 & 0.92 & 0.97 \\
		Method 2 & 1.00 & 1.00 & 1.00 & 0.96 & 0.98 & 1.00 & 1.00 & 1.00 & 1.00\\
		Method 3 & 0.98 & 0.98 & 0.97 & 0.35 & 0.55 & 0.50 & 0.99 & 0.96 & 0.98\\
	\end{tabular}
	\caption{\scriptsize Frequencies that $f_2(x)$ is inside of the corresponding credible interval for the squared exponential Gaussian process prior at given points $x\in\{0.25,0.3188,0.75\}$. Method 1: MMLE empirical Bayes procedure, Method 2: empirical Bayes procedure with $\log n$ blow up factor, Method 3: modified empirical Bayes procedure (MMLE multiplied by $\log n$). From left to right the sample size is $n=100,500,1000$.}
\label{table: regression1}
\end{table}

\begin{table}[!h]
\centering
	\begin{tabular}{c|c|c|c}
		$n =$ & 100 & 500 & 1000\\ \hline
		Method 1 & 0.3956 & 0.2367 & 0.1814\\
		Method 2 & 1.8218 & 1.4711 & 1.2533\\
		Method 3 & 0.7541 & 0.5279 & 0.4262\\
	\end{tabular}
	\caption{\scriptsize Average size of the pointwise credible intervals (i.e. $2q_{0.025}\sqrt{\hat{c}(x,x)}$) for $f_2(x)$. Method 1: MMLE empirical Bayes procedure, Method 2: empirical Bayes procedure with $\log n$ blow up factor, Method 3: modified empirical Bayes procedure (MMLE multiplied by $\log n$). From left to right the sample size is $n=100,500,1000$.}
\label{table: radius}
\end{table}

Next we consider the nonparametric classification problem. Let us assume that we observe the binary random variables $Y_1,Y_2,...,Y_n\in\{0,1\}$, with
\begin{align*}
P(Y_i=1)=p(X_i),\qquad X_i\stackrel{iid}{\sim}U(0,1),
\end{align*}
for some nonparametric function $p(x):\, [0,1]\mapsto [0,1]$. We write $p(x)$ in the form $p(x)= \psi(f(x))$, with $\psi(x)=e^{x}/(1+e^{x})$, for some function $f(x):\,[0,1]\mapsto  \mathbb{R}$. In the Bayesian approach we endow the functional parameter $f(x)$ with a Gaussian process prior with squared exponential kernel. We have again excludes the Mat\'ern kernel from the simulations due to the high computational complexity.

We design similar experiments for the nonparametric classification model as for the nonparametric regression model above, with sample sizes $n=100,200,500$ and $1000$ and the same $f_2$ as above. We plot the pointwise credible intervals for $f_2$ corresponding to the empirical Bayes procedure, with and without a $\log n$ inflation factor, and to the modified empirical Bayes procedure (where the MMLE is multiplied by a $\log n$ factor), see Figure \ref{fig: class}. One can observe that the standard MMLE empirical Bayes procedure produces unreliable uncertainty statements, while by blowing up the credible sets with a logarithmic factor we get overly conservative uncertainty quantifications. These problems are resolved by considering the modified empirical Bayes method, which captures the shape of the underlying functional parameter  better and provides more reliable uncertainty statements. We also collect the empirical estimation of the frequentist coverage probabilities of the underlying functional parameter $f_2(x)$ at points $x=(0.25,0.3188,0.75)$ in Table \ref{table: class1}, underlying the conclusions drawn from the figures above. Moreover, Table \ref{table: radius2} shows the size of the average credible interval.

\begin{figure}[!h]
	\centering
	\includegraphics[scale=0.5]{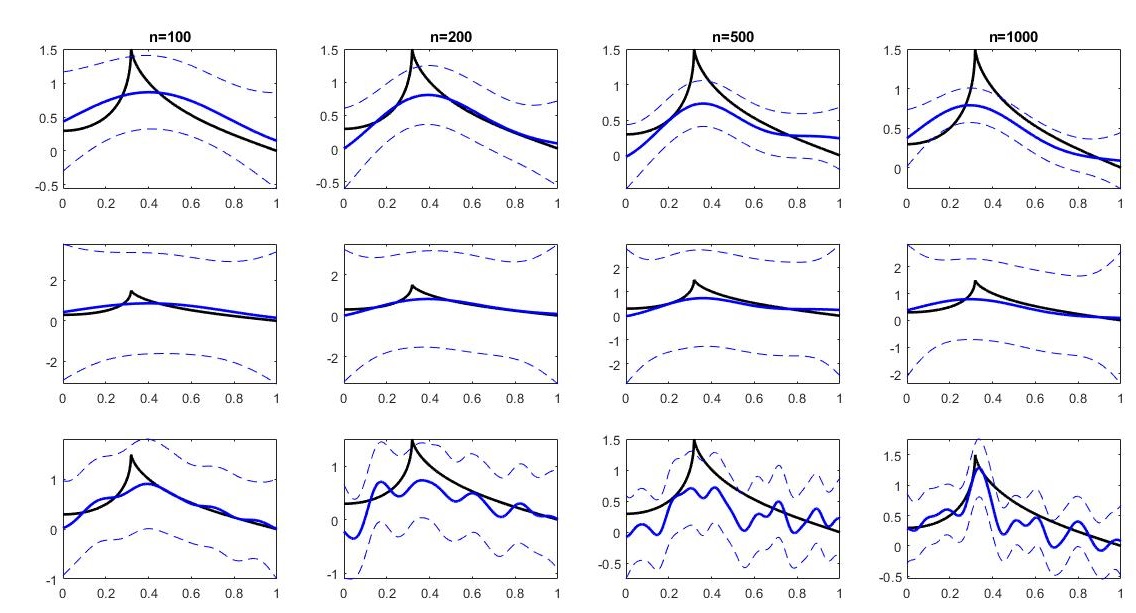}
	\caption{\scriptsize Empirical Bayes credible sets using squared exponential Gaussian process priors in the classification model for the function $f_2$ (drawn in black). The posterior means are drawn by solid blue line, while the $95\%$ pointwise credible intervals by dashed blue curves. In the first row we plotted the MMLE empirical Bayes method, in the second row the MMLE empirical Bayes method with a $\log n$ blow up factor, while in the  the third row the modified MMLE empirical Bayes method. From left to right the sample size is $n=100,200,5000,1000$.}
\label{fig: class}
\end{figure}

\begin{table}[!h]
\centering
	\begin{tabular}{c|c|c|c|c|c|c|c|c|c}
		& \multicolumn{3}{c|}{$x=0.25$}& \multicolumn{3}{c|}{$x=0.3188$}&\multicolumn{3}{c}{$x=0.75$}\\
		$n=$& $100$ & $200$ & $500$& $100$ & $200$ & $500$&100 &200 &500\\ \hline
		Method 1 & 0.90 & 0.90 & 0.89 & 0.29 & 0.16 & 0.12& 0.92 & 0.88 & 0.85\\
		Method 2 & 1.00 & 1.00 & 1.00& 0.98 & 0.98 & 1.00&1.00 & 1.00 & 1.00\\
		Method 3 & 0.91 & 0.94 & 0.95& 0.42 & 0.36 & 0.45& 0.94 & 0.94 & 0.95\\
	\end{tabular}
	\caption{\scriptsize  Frequencies that $f_2(x)$ is inside of the corresponding credible interval for squared exponential Gaussian process prior at given points $x\in\{0.25,0.3188,0.75\}$. Method 1: MMLE empirical Bayes procedure, Method 2: empirical Bayes procedure with $\log n$ blow up factor, Method 3: modified empirical Bayes procedure (MMLE multiplied by $\log n$). From left to right the sample size is $n=100,200,500$.}
\label{table: class1}
\end{table}

\begin{table}[!h]
\centering
	\begin{tabular}{l|c|c|c}
		& n=100 & n=200 & n=500\\ \hline
		Method 1 & 3.2672 & 0.8209  & 0.3485 \\
		Method 2 $\log n$ & 15.0461 & 4.3495 & 2.1661\\
		Method 3  & 3.6777 & 1.2675 & 0.7575\\
	\end{tabular}
	\caption{\scriptsize Average size of the pointwise credible intervals $2q_{0.025}\sqrt{\hat{c}(x,x)}$ for $f_2$ corresponding to the empirical Bayes procedure (first line), to empirical Bayes procedure with $\log n$ blow up factor (second line) and to the modified empirical Bayes procedure (third line). From left to right the sample size is $n=100,200,500$.}
\label{table: radius2}
\end{table}

\section{Discussion}\label{sec: discussion}
We have shown that the MMLE empirical Bayes method for Gaussian process prior with (a slightly modified version of the) squared exponential covariance kernel produces misleading uncertainty statement in context of the Gaussian white noise model. The derived negative results were demonstrated on a simulation study in context of the Gaussian white noise model and extended to the nonparametric regression and classification models as well. Hence we can conclude that one has to be very cautious when applying empirical Bayes methods with squared exponential Gaussian processes for uncertainty quantification as typically they provide misleading confidence statements, due to over-smoothing behaviour of the MMLE. We note that the bad performance of the prior \eqref{def: prior} is not due to the factor $a^{-1}$ in the variance, because similar (but easier) computations show that the prior without the  $a^{-1}$ factor behaves suboptimally as well. Hence the over-smoothing and as a result the misleading uncertainty statements are due to the exponential decay of the eigenvalues of the prior and can not be corrected by simply multiplying the variance with the scaling parameter $a$.

One can compensate the haphazard uncertainty statements by blowing up the credible sets with a $\log n$ factor, however, this approach is not appealing from a practical perspective, as demonstrated in our simulation study as well. Instead we propose to modify the MMLE by multiplying it with $\log n$ to compensate for the over-smoothing. This procedure is less conservative than the previous one and hence provides more information about the uncertainty of the method. One can also consider different covariance kernels, with polynomially decaying eigenvalues, like the Mat\'ern kernel, however, these procedures can be computationally less appealing.

\appendix
\section{Some properties of the MMLE}

\subsection{Deterministic bounds}
As a first step we provide deterministic bounds for the marginal maximum likelihood estimator $\hat{a}_n$ of the rescaling hyper-parameter $a$. Let us introduce the following functions for $a\in [1,\infty)$:
\begin{align}\label{eq: h}
h_n(a,f_0)&:=\frac{1}{\log^2\big(n/a\big)}\sum\limits_{i=1}^{\infty}\frac{n^2ie^{i/a}f_{0,i}^2}{a(ae^{i/a}+n)^2},\\
\label{eq: g}
g_n(a,f_0)&:=\frac{1}{\log^2\big(n/a\big)}\sum\limits_{i=2a}^{\infty}\frac{n^2(i-a)e^{i/a}f_{0,i}^2}{a(ae^{i/a}+n)^2}.
\end{align}
These functions are derived from the expected value of the score function, see Section \ref{sec: consistency a}. Then let us define the deterministic bounds $\underline{a}_n$ and $\overline{a}_n$ for $\hat{a}_n$ with the help of the functions $h_n$ and $g_n$ as 
\begin{align}
\underline{a}_n&:=\sup\{a\in[1,A_n]: g_n(a,f_0)\geq B\log n\},\nonumber\\
\overline{a}_n&:=\sup\{a\in[K_0,A_n]:h_n(a,f_0)\geq b\},\label{def: bounds}
\end{align}
with some $b,B,K_0>0$ to be specified later and $A_n=o(n)$ given in \eqref{def: MMLE}. Then we show that these bounds sandwich $\hat{a}_n$ with high probability.

\begin{theorem}\label{th: consistency a}
The MMLE $\hat{a}_n$ satisfies 
\begin{align}
\inf\limits_{f_0\in\ell_2(M)}P_{0}(\underline{a}_n\leq\hat{a}_n\leq\overline{a}_n)\to 1,
\end{align}
for $\underline{a}_n,\overline{a}_n$ defined in \eqref{def: bounds}.
\end{theorem}
\begin{proof} See Section \ref{sec: consistency a}.
\end{proof}

We also derive an upper bound for $\overline{a}_n$ in case the true function belongs to the hyper-rectangle with regularity hyper-parameter $\beta$.

\begin{proposition}\label{prop: bound a}~For every $\beta\geq\beta_0$ and $\gamma>0$ there exist  $C_{\beta,b,M}, C_{\gamma,b,M}>0$ such that
\begin{align*}
\sup_{f_0\in \Theta^{\beta}(M)}\overline{a}_n&\leq C_{\beta,b,M} n^{1/(1+2\beta)}(\log n)^{-1-1/(1+2\beta)},\\
\sup_{f_0\in A^{\gamma}(M)}\overline{a}_n&\leq  C_{\gamma,b,M}.
\end{align*}
\end{proposition}

\begin{proof}
Let us start with the proof of the first inequality. We show that  for any $b>0$ the inequality $h_n(a,f_0)< b$ holds for $a\geq C_{\beta,b,M} n^{1/(1+2\beta)}(\log n)^{-1-1/(1+2\beta)}$. Let us introduce the notation $I_a\equiv a\log(n/a)$. Note that by using the inequalities $ae^{i/a}+n\geq n$ and $ae^{i/a}+n\geq ae^{i/a}$, for all $a\geq 1$, and the sum of geometric series we get  
\begin{align*}
h_n(a,f_0)&\leq \frac{M}{\log^2(n/a)}\Big(\frac{1}{a}\sum_{i=1}^{I_a}e^{i/a}i^{-2\beta}+\frac{n^2}{a^3}\sum_{i>I_a}e^{-i/a}i^{-2\beta}\Big)\\
&\lesssim\frac{M}{\log^2(n/a)}\Big(I_a^{-2\beta}e^{I_a/a}+\frac{n^2}{a^2}I_a^{-2\beta}e^{-I_a/a}\Big)\\
&\asymp Ma^{-1-2\beta}n\big(\log(n/a)\big)^{-2-2\beta}.
\end{align*}
For $A_n\geq a\geq C_{\beta,b,M}n^{1/(1+2\beta)}(\log n)^{-1-1/(1+2\beta)}$ the preceding display is bounded by a multiple of $MC_{\beta,b,M}^{-1-2\beta}$. Then for sufficiently large choice of the constant $C_{\beta,b,M}$, we get that $h_n(a,f_0)< b$ for any $a\geq C_{\beta,b,M} n^{1/(1+2\beta)}(\log n)^{-1-1/(1+2\beta)}$.

The proof of the second inequality of the statement goes similarly, i.e.  we prove that for $a\geq C_{\gamma,b,M}$ we have $h_n(a,f_0)< b$. Note that by the sum of geometric series we get for every $a\geq 1/\gamma$ 
\begin{align*}
h_n(a,f_0)&\leq \frac{M}{\log^2(n/a)}\Big(\frac{1}{a}\sum_{i=1}^{I_a}ie^{i/a}e^{-2\gamma i}+\frac{n^2}{a^3}\sum_{i>I_a}ie^{-i/a}e^{-2\gamma i} \Big)\\
&\leq \frac{M}{a\log^2(n/a)}\sum_{i=1}^{\infty}ie^{-\gamma i} 
\leq \frac{M}{a(1-e^{-\gamma})^2\log^2(n/a)},
\end{align*}
which is bounded from above by arbitrarily small $b$ for sufficiently large choice of the constant $C_{\gamma,b,M}$.
\end{proof}

Finally we show that under the polished tail condition the deterministic bounds $\underline{a}_n, \overline{a}_n$ are close to each other.

\begin{lemma}\label{th: interval a polish tail}~For every $L_0,\rho,N_0\geq 1$ we have
\begin{align*}
\sup\limits_{f_0\in\Theta_{pt}(L_0,N_0,\rho)}\frac{\overline{a}_n\log(n/\overline{a}_n)}{\underline{a}_n\log(n/\underline{a}_n)}\leq K\log^2 n,
\end{align*}
with $K=8.1e^4\rho^2L_0 B /b$ for $n$ large enough.
\end{lemma}

\begin{proof}
First of all note that since $\underline{a}_n\leq\overline{a}_n$, there is nothing to prove in the trivial cases $\underline{a}_n=A_n$ or $\overline{a}_n=K_0$. Hence $h_n(\overline{a}_n,f_0)\leq b$ and $g_n(a,f_0)< B\log n$, for all $a>\underline{a}_n$, hold. Furthermore assume that $\underline{a}_n\leq \rho^{-2}\overline{a}_n$, else the statement is trivial.

Let us divide the interval $[\rho^j,\rho^{j+1})$ into sub-intervals $[\rho^{j+k/\lceil\log n\rceil}, \rho^{j+(k+1)/\lceil\log n\rceil})$, $k=0,1,...\lceil\log n\rceil-1$, and introduce the notation 
$$ k_j=\argmax_{k=0,...,\lceil\log n\rceil-1} \vartheta_{j,k},\quad\text{where } \vartheta_{j,k}=\sum_{i=\rho^{j+k/\lceil\log n\rceil}}^{\rho^{j+(k+1)/\lceil\log n\rceil}}f_{0,i}^2,$$
with the notational convenience $\sum_{i=a}^b c_i=\sum_{i=\lceil a\rceil}^{\lfloor b\rfloor}c_i$, applied later on as well.

Then by the polished tail condition
$$ \sum_{i=\rho^{j}}^{\infty}f_{0,i}^2\leq L_0 \sum_{i=\rho^{j}}^{\rho^{j+1}}f_{0,i}^2\leq L_0 \log (n) \vartheta_{j,k_j},$$
for $j\geq \log_\rho N_0$. Note that for every $a>0$ there exists an $\tilde{a}\in(a,\rho^2 a)$ such that 
\begin{align}
I_{\tilde{a}}\equiv \tilde{a}\log(n/\tilde{a})\in\big[ \rho^{j+k_j/ \lceil\log n\rceil},\rho^{j+(k_j+1)/\lceil\log n\rceil} \big)\label{def: tilde:a}
\end{align}
for some $j\in\mathbb{N}$ and let us denote this $j$ by $J_{\tilde{a}}$. Then
$$\sum_{i=e^{-1/\log n}I_{\tilde{a}}}^{e^{1/\log n}I_{\tilde{a}}}f_{0,i}^2\geq\vartheta_{J_{\tilde{a}},k_{J_{\tilde{a}}}}.$$
Let us take any $a_1\leq \rho^{-2} a_2$ and denote by $\tilde{a}_1\in (a_1,\rho^2 a_1)$ the value satisfying \eqref{def: tilde:a}. Then by the previous inequalities and noting that $\exp\{ e^{1/\log n}\log(n/a)\}\leq \exp\{(1+2/\log n)\log(n/a)\}\leq e^2n/a$, for $n\geq e$, using that $e^x\leq 1+2x$ for $x\in[0,1]$,
\begin{align*}
\frac{h_n(a_2,f_0)}{h_n(\tilde{a}_1,f_0)}&\leq \frac{\tilde{a}_1\log^2(\frac{n}{\tilde{a}_1})}{a_2\log^2(\frac{n}{a_2})}4e^2\frac{\sum_{i=1}^{I_{\tilde{a}_1}} ie^{i/a_2}f_{0,i}^2+\sum_{i= I_{\tilde{a}_1}}^{I_{a_2}}  ie^{i/a_2}f_{0,i}^2+\frac{n^2}{a_2^2}\sum_{i= I_{a_2}}^{\infty}ie^{-i/a_2}f_{0,i}^2}{\sum_{i=1}^{e^{1/\log n} I_{\tilde{a}_1}} ie^{i/\tilde{a}_1}f_{0,i}^2}\\
&\leq \frac{\tilde{a}_1\log^2(\frac{n}{\tilde{a}_1})}{a_2\log^2(\frac{n}{a_2})}4e^2\Big(1+\frac{\sum_{i=I_{\tilde{a}_1}}^{I_{a_2}} ie^{i/a_2} f_{0,i}^2+n\log(\frac{n}{a_2})\sum_{i= I_{a_2}}^{\infty}f_{0,i}^2}{\sum_{i=e^{-1/\log n}I_{\tilde{a}_1}}^{e^{1/\log n}I_{\tilde{a}_1}} ie^{i/\tilde{a}_1}f_{0,i}^2}\Big).
\end{align*}
Since $ie^{i/\tilde{a}_1} > e^{-2}n\log(n/\tilde{a}_1)$ for $i\geq e^{-1/\log n}I_{\tilde{a}_1}$, and $ie^{i/a_2} \leq n\log(n/a_2)$ for $i\leq I_{a_2}$, we can see that
$$\frac{\sum_{i=I_{\tilde{a}_1}}^{I_{a_2}} ie^{i/a_2} f_{0,i}^2+n\log(\frac{n}{a_2})\sum_{i= I_{a_2}}^{\infty}f_{0,i}^2}{\sum_{i=e^{-1/\log n}I_{\tilde{a}_1}}^{e^{1/\log n} I_{\tilde{a}_1}} ie^{i/\tilde{a}_1}f_{0,i}^2}\leq e^{2}\frac{\log(\frac{n}{a_2})}{\log(\frac{n}{\tilde{a}_1})}\frac{\sum_{i=I_{\tilde{a}_1}}^{\infty}f_{0,i}^2}{\sum_{i=e^{-1/\log n}I_{\tilde{a}_1}}^{e^{1/\log n}I_{\tilde{a}_1}} f_{0,i}^2}.$$
Moreover, since 
$$\sum_{i=I_{\tilde{a}_1}}^{\infty}f_{0,i}^2\leq L_0\log(n)\vartheta_{J_{\tilde{a}_1},k_{J_{\tilde{a}_1}}}\leq L_0\log(n)\sum_{i=e^{-1/\log n}I_{\tilde{a}_1}}^{e^{1/\log n}I_{\tilde{a}_1}}f_{0,i}^2,$$
combined with the preceding computations we get that 
\begin{align}
 \frac{h_n(a_2,f_0)}{h_n(\tilde{a}_1,f_0)}\leq 4e\frac{\tilde{a}_1\log^2(\frac{n}{\tilde{a}_1})}{a_2\log^2(\frac{n}{a_2})}\Big(1+L_0e^{2}\log(n)\frac{\log(\frac{n}{a_2})}{\log(\frac{n}{\tilde{a}_1})}\Big).\label{eq: help:dist:bounds}
\end{align}
Furthermore, let us note that for any $\underline{a}_n<a\leq A_n$
$$h_n(a,f_0)\leq 2g_n(a,f_0)+\frac{2e^2}{\log^2(\frac{n}{a})}\sum\limits_{i=1}^{2a}f^2_{0,i}\leq 2B\log(n)+o(1).$$
 Then by taking $a_1=\underline{a}_n$, $\tilde{a}_1\in(\underline{a}_n,\rho^2\underline{a}_n)$, and $a_2=\overline{a}_n$ in $\eqref{eq: help:dist:bounds}$ we get that
\begin{align*}
\frac{b}{2B\log (n)+o(1)} \leq \frac{h_n(\overline{a}_n,f_0)}{h_n(\tilde{a}_1,f_0)}&\leq 4e^4(1+o(1))L_0\log (n) \frac{\tilde{a}_1\log(\frac{n}{\tilde{a}_1})}{\overline{a}_n\log(\frac{n}{\overline{a}_n})}\\
& \leq 4e^4\rho^2(1+o(1))L_0\log (n) \frac{\underline{a}_n\log(\frac{n}{\underline{a}_n})}{\overline{a}_n\log(\frac{n}{\overline{a}_n})}.
\end{align*}
After rearranging the preceding inequality we arrive to our statement.
\end{proof}

\subsection{Contraction rates}\label{sec: contraction}
In this section we provide the contraction rate results both for the empirical and hierarchical Bayes procedures.First we show that the empirical Bayes method achieves the (up to a logarithmic factor) optimal minimax contraction rate around the truth for unknown regularity hyper-parameter $\beta>0$.

\begin{theorem}\label{thm: eb:contraction}
The maximum marginal likelihood empirical Bayes posterior corresponding to the prior \eqref{def: prior} achieves the minimax adaptive contraction rate (up to a logarithmic factor), i.e. for given $M,\beta_0>0$ we have
\begin{align}
\label{eq: eb:contraction expect}
\sup_{\beta\geq \beta_0}\sup\limits_{f\in\Theta^{\beta}(M)}E_0[\Pi_{\hat{a}_n}(\|f-f_0\|_2\geq M_n (n/\log^2n)^{-\beta/(1+2\beta)}|Y)]\to 0,
\end{align}
for any sequence $M_n$ tending to infinity.
\end{theorem}
\begin{proof}
See Section \ref{sec: eb:contraction}.
\end{proof}

Using our findings on the empirical Bayes method we can extend the results on the hierarchical Bayes method, derived in the literature (where typically inverse gamma hyper-prior was considered), by allowing other, more general choices of the hyper-prior distribution as well. 
\begin{theorem}\label{thm: hb:contraction}
Let us assume that the hyper-prior $\pi$ satisfies Assumption \ref{assump: HB}. Then the corresponding hierarchical Bayes posterior achieves the minimax contraction rate (up to a logarithmic factor), i.e.  for given $\beta_0,M>0$ we have
\begin{align}
\sup_{\beta\geq\beta_0}\sup\limits_{f\in\Theta^{\beta}(M)}E_0[\Pi(\|f-f_0\|_2\geq M_n (n/\log^2n)^{-\beta/(1+2\beta)}|Y)]\to 0,
\end{align}
for some arbitrary sequence $M_n$ tending to infinity.
\end{theorem}
\begin{proof}
See Section \ref{sec: hb:contraction}.
\end{proof}
The above results are of interest in their own right, but our main focus lies on the reliability of Bayesian uncertainty quantification resulting both from the hierarchical and the empirical Bayes procedures, hence we have deferred the contraction rate results to the appendix.

\subsection{Proof of Theorem \ref{thm: eb:contraction}}\label{sec: eb:contraction}
Let us introduce the shorthand notation $\varepsilon_n:=n^{-\beta/(1+2\beta)}(\log n)^{2\beta/(1+2\beta)}$. In view of Markov's inequality and Theorem \ref{th: consistency a}, for every $\beta>0$
\begin{align}
\label{eq : contraction mark ineq}
\sup\limits_{f_0\in\Theta^{\beta}(M)}E_0[\Pi_{\hat{a}_n}(\|f-f_0\|_2\geq M_n\varepsilon_n|Y)]\leq\frac{1}{M_n^2\varepsilon_n^2}\sup\limits_{f_0\in\Theta^{\beta}(M)}E_0[\sup\limits_{a\in[\underline{a}_n,\overline{a}_n]}R_n(a)]+o(1),
\end{align}
where
$$ R_n(a)=\int \|f-f_0\|_2^2\Pi_a(df|Y) $$
is the posterior risk. We show below that both
\begin{align}\label{eq: bound risk}
\sup\limits_{f_0\in\Theta^{\beta}(M)}\sup\limits_{a\in[\underline{a}_n,\overline{a}_n]}E_0[R_n(a)]=O(\varepsilon_n^2)\qquad \text{ and}\\
\label{eq: posterior risk}
\sup\limits_{f_0\in\Theta^{\beta}(M)}E_0[\sup\limits_{a\in[\underline{a}_n,\overline{a}_n]}|R_n(a)-E_0(R_n(a))|]=o(\varepsilon_n^2)
\end{align}
hold, which results in that the right-handside of (\ref{eq : contraction mark ineq}) vanishes as $n\to\infty$, concluding the proof of Theorem \ref{thm: eb:contraction}.

\subsubsection{Bound for the expected posterior risk \eqref{eq: bound risk}}
First, note that by elementary computations
\begin{align*}
 R_n(a)=\sum\limits_{i=1}^{\infty}(\hat{f}_{a,i}-f_{0,i})^2+\sum\limits_{i=1}^{\infty}\frac{1}{ae^{i/a}+n},
\end{align*}
where $\hat{f}_{a,i}=n(ae^{i/a}+n)^{-1}Y_i$ is the $i$th coefficient of the posterior mean. Therefore the expectation of $R_n(a)$ is given by
\begin{align}
E_0  R_n(a)= \sum\limits_{i=1}^{\infty}\frac{a^2e^{2i/a}}{(ae^{i/a}+n)^2}f_{0,i}^2+\sum\limits_{i=1}^{\infty}\frac{n}{(ae^{i/a}+n)^2}+\sum\limits_{i=1}^{\infty}\frac{1}{ae^{i/a}+n}. \label{eq:help:th1_1}
\end{align}
Note that the second and third terms do not contain $f_0$, and that the second term is bounded by the third. By Lemma \ref{lem: bound sum} (with $r=0$ and $l=1$) and Proposition \ref{prop: bound a} the latter is further bounded for $a\leq\overline{a}_n$  by a multiple of
$$ \frac{a}{n}\log\big(\frac{n}{a}\big)\leq\frac{\overline{a}_n}{n}\log\big(\frac{n}{\overline{a}_n}\big)\lesssim n^{-2\beta/(1+2\beta)}(\log n)^{-1/(1+2\beta)},$$
since the function $a\mapsto a\log (n/a)$ is monotone increasing for $a\leq n/e$. It remained to deal with the first term on the right hand side of \eqref{eq:help:th1_1}, which we divide into three parts and show that each of the parts have the stated order. First note that for $f_0\in\Theta^{\beta}(M)$
$$ \sum_{i=(n/\log^2 n)^{1/(1+2\beta)}}^{\infty}\frac{a^2e^{2i/a}}{(ae^{i/a}+n)^2}f_{0,i}^2\leq\sum_{i=(n/\log^2 n)^{1/(1+2\beta)}}^{\infty}Mi^{-1-2\beta}\leq\frac{M}{2\beta}(n/\log^2 n)^{-2\beta/(1+2\beta)}.$$
Next note that for $a\leq\overline{a}_n$, in view of Proposition \ref{prop: bound a},
$$ \sum_{i=1}^{2a}\frac{a^2e^{2i/a}}{(ae^{i/a}+n)^2}f_{0,i}^2\leq\sum_{i=1}^{2a}\frac{a^2e^{2i/a}}{n^2}f_{0,i}^2\leq\frac{a^2e^4}{n^2}\sum_{i=1}^{2a}f_{0,i}^2\lesssim\frac{\overline{a}_n^2}{n^2}\lesssim n^{-\frac{4\beta}{1+2\beta}}(\log n)^{-2-\frac{2}{1+2\beta}}. $$
Furthermore, notice that the maximum of the function $i\mapsto e^{i/a}/(i-a)$ over $[2a,I_a]$ is attained at $i=I_a$, because the function is increasing for $i>2a$ and $n>0$. Besides, for $a>\underline{a}_n$ we have $g_n(a,f_0)< B\log n$, hence for any $\underline{a}_n< a\leq\bar{a}_n$
\begin{align*}
\sum_{i=2a}^{I_a}\frac{a^2e^{2i/a}}{(ae^{i/a}+n)^2}f_{0,i}^2&\leq\frac{a}{n}\frac{\log^2(n/a)}{(\log(n/a)-1)}\sum_{i=2a}^{I_a}\frac{n^2e^{i/a}(i-a)}{a\log^2(n/a)(ae^{i/a}+n)^2}f_{0,i}^2\\
&\lesssim \overline{a}_nn^{-1}\log(n/\overline{a}_n)g_n(a,f_0)\\
&\leq \overline{a}_n n^{-1}\log^2n\lesssim n^{-2\beta/(1+2\beta)}(\log n)^{2\beta/(1+2\beta)},
\end{align*}
where the last inequality follows from Proposition \ref{prop: bound a}.

It remained to deal with the terms between $I_{\underline{a}_n}=\underline{a}_n\log(n/\underline{a}_n)$ and $(n/\log^2 n)^{1/(1+2\beta)}$. Let $J=J(n)$ be the smallest integer such that 
$$\big(1+\frac{1}{\log n}\big)^J\underline{a}_n\log\big(\frac{n}{\underline{a}_n}\big)\geq (n/\log^2 n)^{1/(1+2\beta)}$$
and let
$$ n_j:=\big(1+\frac{1}{\log n}\big)^j I_{\underline{a}_n}.$$
Note that the sequence $n_j$ is increasing. For notational convenience, we also introduce $b_j$ such that $b_j e^{n_j/b_j}=n$ and $b_j<n_j$. Now we have for any $a\geq 1$
\begin{align}
\sum_{i=I_{\underline{a}_n}}^{(n/\log^2 n)^{1/(1+2\beta)}}\frac{a^2e^{2i/a}}{(ae^{i/a}+n)^2}f_{0,i}^2&\leq\sum_{j=0}^{J-1}\sum_{i=n_j}^{n_{j+1}}f_{0,i}^2\nonumber\\
&\leq4e^2\sum_{j=0}^{J-1}\sum_{i=n_j}^{n_{j+1}}\frac{nb_je^{i/b_j}}{(b_{j}e^{i/b_{j}}+n)^2}f_{0,i}^2.\label{eq: risk:help1}
\end{align}
By elementary computations we get that $b_j\asymp n_j/\log n_j$, therefore $\eqref{eq: risk:help1}$ is further bounded by constant times
$$ \frac{1}{n}\sum_{j=0}^{J-1}\frac{1}{\log n_j}\sum_{i=n_j}^{n_{j+1}}\frac{n^2(i-b_j)e^{i/b_j}}{(b_je^{i/b_j}+n)^2}f_{0,i}^2\leq\frac{1}{n}\sum_{j=0}^{J-1}\frac{b_j\log^2 n}{\log n_j}g_n(b_j,f_0).$$
Since $b_j\geq\underline{a}_n$ we have $g_n(b_j,f_0)\leq B\log n$ for all $j\geq 0$. Then by the sum of geometric series we get that
\begin{align*}
\frac{1}{n}\sum_{j=0}^{J-1}\frac{n_j}{\log^2 n_j}\log^3 n
\lesssim\frac{\log n}{n}\frac{I_{\underline{a}_n}(1+1/\log n)^{J}}{1/\log n} \leq n^{-2\beta/(1+2\beta)}(\log n)^{2-2/(1+2\beta)},
\end{align*}
concluding the proof of assertion \eqref{eq: bound risk}.

\subsubsection{Bound for the centered posterior risk \eqref{eq: posterior risk}}
Note that 
\begin{align*}
&R_n(a)-E_0 R_n(a)=\mathbb{V}(a)/n-2\mathbb{W}(a)/\sqrt{n}, \quad\text{where}\\
&\mathbb{V}(a)=n^2\sum\limits_{i=1}^{\infty}\frac{1}{(ae^{i/a}+n)^2}(Z_i^2-1) \mbox{, and } \mathbb{W}(a)=n\sum\limits_{i=1}^{\infty}\frac{ae^{i/a}f_{0,i}}{(ae^{i/a}+n)^2}Z_i.
\end{align*}
Therefore it is sufficient to show that
\begin{align*}
E_0\big(\sup\limits_{a\in[\underline{a}_n,\overline{a}_n]}|\mathbb{V}(a)|/n\big)&\lesssim n^{-2\beta/(1+2\beta)}(\log n)^{-1/(1+2\beta)},\\ 
E_0\big(\sup\limits_{a\in[\underline{a}_n,\overline{a}_n]} |\mathbb{W}(a)/\sqrt{n}|\big)&\lesssim n^{-2\beta/(1+2\beta)}.
\end{align*}
We deal with the two processes above, separately.

For the process $\mathbb{V}$, Corollary 2.2.5 in \cite{vdv:wellner:96} implies that
$$ E_0[\sup\limits_{a\in[\underline{a}_n,\overline{a}_n]}|\mathbb{V}(a)|]\lesssim\sup\limits_{a\in[\underline{a}_n,\overline{a}_n]}\sqrt{V_0(\mathbb{V}(a))}+\int\limits_0^{diam_n}\sqrt{N(\varepsilon,[\underline{a}_n,\overline{a}_n],d_n)}d\varepsilon,$$
where $d_n^2(a_1,a_2)=V_0(\mathbb{V}(a_1)-\mathbb{V}(a_2))$, $diam_n$ is the $d_n$-diameter of $[\underline{a}_n,\overline{a}_n]$ and $N(\varepsilon,B,d_n)$ the covering number of the set $B$ with $\varepsilon$-radius balls relative to the $d_n$ semi-metric. The variance of $\mathbb{V}(a)$ is equal to
$$V_0(\mathbb{V}(a))=2n^4\sum\limits_{i=1}^{\infty}\frac{1}{(ae^{i/a}+n)^4}$$
since $V(Z_i^2)=2$. Using Lemma \ref{lem: bound sum} (with $r=0$ and $l=4$) we can conclude that the variance of $\mathbb{V}(a)$ is bounded from above by a multiple of $a\log(n/a)$, hence $diam_n \lesssim \sqrt{\overline{a}_n\log n}$. In view of Lemma \ref{lem: bound v w}, the distance $d_n(a_1,a_2)$ is bounded from above by a multiple of $|a_1-a_2|\log^{3/2} n$, hence the interval $[\underline{a}_n,\overline{a}_n]$ can be covered with constant times $\overline{a}_n \varepsilon^{-1}\log^{3/2} n$ amount of $\varepsilon$-balls relative to the $d_n$ semi-metric. In view of the above computation and Proposition \ref{prop: bound a}
$$E_0[n^{-1}\sup\limits_{a\in[\underline{a}_n,\overline{a}_n]}|\mathbb{V}(a)|]\lesssim(\overline{a}_n/n)\log n\lesssim n^{-2\beta/(1+2\beta)}(\log n)^{-1/(1+2\beta)}.$$

The process $\mathbb{W}$ can be dealt with similarly to $\mathbb{V}$. The main difference is the bounding of the variance of $\mathbb{W}$, which we describe in details. First note that
$$ V_0\Bigg(\frac{\mathbb{W}(a)}{\sqrt{n}}\Bigg)=n\sum\limits_{i=1}^{\infty}\frac{a^2e^{2i/a}}{(ae^{i/a}+n)^4}f_{0,i}^2.$$
Let us split the sum at $I_a$ and by applying the inequality $ae^{i/a}+n\geq n$, we get
$$ n\sum\limits_{i=1}^{I_a}\frac{a^2e^{2i/a}}{(ae^{i/a}+n)^4}f_{0,i}^2\leq \frac{1}{n^3}\sum\limits_{i=1}^{I_a}a^2e^{2i/a}f_{0,i}^2\lesssim \frac{\|f_0\|_2^2}{n}. $$
Then by noting that the function $i\mapsto e^{i/a}/((i-a)(ae^{i/a}+n)^2)$ is decreasing on $[I_a,\infty)$, recalling that $g_n(a,f_0)\leq B\log n$, for all $a\geq \underline{a}_n$, and in view of Proposition \ref{prop: bound a}
\begin{align*}
n\sum\limits_{i= I_a}^\infty\frac{a^2e^{2i/a}}{(ae^{i/a}+n)^4}f_{0,i}^2&\leq \frac{a\log^2(n/a)}{4n^2(\log(n/a)-1)}\sum\limits_{i=I_a}^{\infty}\frac{n^2(i-a)e^{i/a}}{a\log^2(n/a)(ae^{i/a}+n)^2}f_{0,i}^2\\
&\lesssim an^{-2} \log(n/a) g_n(a,f_0)\lesssim \overline{a}_n n^{-2}\log^2 n\\
&=O( n^{-(1+4\beta)/(1+2\beta)}(\log n)^{2\beta/(1+2\beta)}),
\end{align*}
hence $\diam_n=O(n^{-\frac{1/2+2\beta}{1+2\beta}}(\log n)^{\frac{\beta}{1+2\beta}}))$. Then in view of Lemma \ref{lem: bound v w} the covering number of the interval $[\underline{a}_n,\overline{a}_n]$ is bounded by $\varepsilon^{-1}(\overline{a}_n/\sqrt{n})\log n$ with respect to the semi-metric $d_n(a_1,a_2)=V_0\big(\mathbb{W}(a_1)/\sqrt{n}-\mathbb{W}(a_2)/\sqrt{n}\big)$ and the rest of the proof goes as above.

\subsubsection{Bounds for the semimetrics associated to $\mathbb{V}$ and $\mathbb{W}$}
\begin{lemma}\label{lem: bound v w}~For any $1\leq a_1\leq a_2$ and $f_0\in\ell_2(M)$ we have
\begin{align*}
V_0\Big(\mathbb{V}(a_1)-\mathbb{V}(a_2)\Big)\lesssim (a_1-a_2)^2\log^3 n,\\
V_0\Big(\mathbb{W}(a_1)-\mathbb{W}(a_2)\Big)\lesssim (a_1-a_2)^2\log^2 n,
\end{align*}
with constants only depending on $M$.
\end{lemma}

\begin{proof}~The left-hand side of the first inequality is equal to
$$ n^4\sum\limits_{i=1}^{\infty}(\phi_i(a_1)-\phi_i(a_2))^2V(Z_i^2), $$
where $\phi_i(a)=(ae^{i/a}+n)^{-2}$. The square of the derivative of $\phi_i$ is given by $\phi_i'(a)^2=4\phi_i(a)^3e^{2i/a}(i-a)^2/a^2$, hence in view of Lemma \ref{Lem: Diff} the preceding display is bounded above by a multiple of 
\begin{align*}
(a_1-a_2)^2n^4\sup_{a\in[a_1,a_2]}\sum_{i=1}^{\infty}\frac{e^{2i/a}(i-a)^2}{a^2(ae^{i/a}+n)^6}&\leq (a_1-a_2)^2n^4\sup_{a\in[a_1,a_2]}\sum_{i=1}^{\infty}\frac{e^{2i/a}(i^2+a^2)}{a^2(ae^{i/a}+n)^6}\\
&\lesssim (a_1-a_2)^2\sup_{a\in[a_1,a_2]}\frac{\log(n/a)}{a}\Big(1+\log^2\big(\frac{n}{a}\big)\Big)
\end{align*}
with the help of Lemma \ref{lem: bound term} (first with $m=2$ and then with $m=0$), and Lemma \ref{lem: bound sum} (with $r=1$ and $l=4$).

We next consider the process $\mathbb{W}(a)$. The left-hand side of the second inequality in the statement of the lemma is equal to
$$ n^2\sum\limits_{i=1}^{\infty}(\phi_i(a_1)-\phi_i(a_2))^2f_{0,i}^2V_0(Z_i),$$
with $\phi_i(a)=ae^{i/a}/(ae^{i/a}+n)^2$. Note that $|\phi_i'(a)|\leq (i+a)a^{-2}\phi_i(a)$, hence in view of Lemma \ref{lem: bound term} (first with $m=2$ and then with $m=0$) and Lemma \ref{Lem: Diff} the preceding display is bounded by
\begin{align*}
 4(a_1-a_2)^2n^2\sup_{a\in[a_1,a_2]}&\frac{1}{a^2}\sum_{i=1}^{\infty}\frac{e^{2i/a}(i^2/a^2+1)}{(ae^{i/a}+n)^4}f_{0,i}^2\\
&\leq 4(a_1-a_2)^2\sup\limits_{a\in[a_1,a_2]}\frac{1}{a^2}\big(\log^2(n/a)+1\big)\|f_{0}\|_2^2, 
\end{align*}
concluding the proof of the lemma.
\end{proof}

\section{Proof of the empirical Bayes part of Theorem \ref{thm: polish tail coverage}}\label{sec: polish tail coverage}
First note that we get the empirical Bayes credible set by plugging in the estimator $\hat{a}_n$ into the credible set $\hat{C}(a)$ with fixed hyper-parameter $a$, given in \eqref{def: credible_a}. The proof of the statement is then based on the deterministic bounds for the MMLE $\hat{a}_n$ derived in Theorem \ref{th: consistency a} and their distance investigated in Lemma \ref{th: interval a polish tail}. 

Note that $f_0\in \hat{C}_n(L\log n)$ if and only if $\|f_0-\hat{f}\|_2\leq L\log (n) r_{\alpha}$. Therefore by triangle inequality it is sufficient to verify that 
\begin{align}
\|W(\hat{a}_n)\|_2 \leq L\log (n) r_{\alpha}(\hat{a}_n)-\|B(\hat{a}_n,f_0)\|_2\label{eq: help_coverage}
\end{align}
 holds with high probability, where $W(a)=\hat{f}_a-E_0\hat{f}_a$ and $B(a,f_0)=E_0\hat{f}_a-f_0$ are the centered posterior mean and the bias of the posterior mean for fixed hyperparameter $a>0$, respectively. Note that the $i$th coefficient of these vectors take the form
\begin{align*}
W_i(a) = \frac{n(Y_i-f_{0,i})}{ae^{i/a}+n},\quad\text{and}\quad B_i(a,f_0) = \frac{ae^{i/a}f_{0,i}}{ae^{i/a}+n}.
\end{align*}
We prove below that there exist constants $C_1,C_2,C_3>0$ such that
\begin{align}\label{eq: bound radius}
\inf\limits_{\underline{a}_n\leq a \leq\overline{a}_n}r^2_{\alpha}(a)\geq C_1\frac{\underline{a}_n}{n}\log\big(\frac{n}{\underline{a}_n}\big),\\
\label{eq: bound variance}
\inf\limits_{f_0\in\Theta_{pt}(L_0,N_0,\rho)}P_{0}(\sup\limits_{\underline{a}_n\leq a \leq\overline{a}_n}\|W(a)\|_2^2\leq C_2\frac{\underline{a}_n}{n}\log\big(\frac{n}{\underline{a}_n}\big)\log^2n)\to 1,\\
\label{eq: bound bias}
\sup\limits_{\underline{a}_n\leq a \leq\underline{a}_n}\|B(a,f_0)\|_2^2\leq C_3\frac{\underline{a}_n}{n}\log^2\big(\frac{n}{\underline{a}_n}\big)\log(n).
\end{align}
Hence in view of Theorem \ref{th: consistency a} assertion $\eqref{eq: help_coverage}$ holds with probability tending to one for large enough choice of $L$, under the polished tail assumption.

Proof of (\ref{eq: bound radius}): The radius $r_{\alpha}(a)$, given in  \eqref{def: credible_a}, is defined as $P(U_n(a) <r_{\alpha}^2(a))=1-\alpha$ with $U_n(a) := \sum\limits_{i=1}^{\infty}\frac{1}{ae^{i/a}+n}Z_i^2$, where $Z_i$'s are iid $N(0,1)$. We show below that
\begin{align}\label{eq: bound expectation radius}
\liminf\limits_{n\to\infty}\inf\limits_{a\in[\underline{a}_n,\overline{a}_n]}E\Big[\frac{n U_n(a)}{a\log(n/a)}\Big]>\frac{1}{2},\\
\label{eq: bound randomness radius}
E\Bigg[\sup\limits_{a\in[\underline{a}_n,\overline{a}_n]}\frac{n|U_n(a)-E[U_n(a)]|}{a\log(n/a)}\Bigg]\to 0.
\end{align}
Then by Markov's inequality with probability tending to one we have
\begin{align*}
\inf_{a\in[\underline{a}_n,\overline{a}_n]}\frac{n U_n(a)}{a\log(n/a)}>1/3,
\end{align*}
hence \eqref{eq: bound radius} follows from the definition of $r_{\alpha}(a)$.

Assertion \eqref{eq: bound expectation radius} follows as
$$E[U_n(a)]\geq\sum\limits_{i=1}^{I_a}\frac{1}{ae^{i/a}+n}\geq\frac{I_a}{2n}\geq\frac{a}{2n}\log\big(\frac{n}{a}\big).$$

To verify (\ref{eq: bound randomness radius}), it suffices by Corollary 2.2.5 in \cite{vdv:wellner:96} (applied with $\psi(x)=x^2$) to show that there exist $K_1,K_2>0$ such that for any $a\in [\underline{a}_n,\overline{a}_n]$ 
\begin{align}\label{eq: bound var radius}
V\Big(\frac{n U_n(a)}{a\log(n/a)}\Big)\leq K_1\frac{1}{a\log (n/a)},\\
\label{eq: bound entropy radius}
\int\limits_{0}^{diam_n}\sqrt{N(\varepsilon, [\underline{a}_n,\overline{a}_n],d_n)}d\varepsilon\leq  \sqrt{A_n/n}=o(1),
\end{align}
where $d_n$ is the semimetric defined by $d_n^2(a_1,a_2):= V\Big(\frac{nU_n(a_1)}{a_1\log(n/a_1)}-\frac{nU_n(a_2)}{a_2\log(n/a_2)}\Big)$, $diam_n$ is the diameter of the interval $ [\underline{a}_n,\overline{a}_n]$ relative to $d_n$ and $N(\varepsilon,S,d_n)$ is the minimal number of $d_n$-balls of radius $\varepsilon$ needed to cover the set $S$.

First note that in view of Lemma \ref{lem: bound sum} (with $r=0$ and $l=2$) we have 
\begin{align*}
V\Big( \frac{nU_n(a)}{a\log(n/a)} \Big)=\frac{2n^2}{a^2\log^2(n/a)}\sum\limits_{i=1}^{\infty}\frac{1}{(ae^{i/a}+n)^2}\lesssim \frac{1}{a\log(n/a)}.
\end{align*}
As a consequence one can see that $diam_n\lesssim(\underline{a}_n \log(n/\underline{a}_n))^{-1/2}$. By Lemma \ref{lem: bound u}, $d_n(a_1,a_2)\lesssim a_1^{-3/2}\log^{1/2}(n/a_1) n^{-1}|a_1-a_2|$, hence 
$$N(\varepsilon,[\underline{a}_n,\overline{a}_n],d_n)\lesssim \varepsilon^{-1}\log^{1/2}(n/\underline{a}_n)\underline{a}^{-3/2}_n\overline{a}_n/n.$$
 Therefore one can conclude that
\begin{align*}
 \int\limits_{0}^{diam_n}\sqrt{N(\varepsilon, [\underline{a}_n,\overline{a}_n],d_n)}d\varepsilon
=\frac{\overline{a}_n^{1/2}\log^{\frac{1}{4}}(n/\underline{a}_n)}{\underline{a}^{3/4}_n n^{1/2}}\int_0^{C (\underline{a}_n \log(\frac{n}{\underline{a}_n}))^{-\frac{1}{2}}}\varepsilon^{-\frac{1}{2}}d\varepsilon\lesssim  \sqrt{A_n/n}.
\end{align*}

Proof of (\ref{eq: bound variance}): The variable $\|W(a)\|^2_2$ is distributed as $\sum\limits_{i=1}^{\infty}\frac{n}{(ae^{i/a}+n)^2}Z_i^2$, with $Z_i\stackrel{iid}{\sim} N(0,1)$. Observe that
$$E_0[\|W(a)\|^2_2]=\sum\limits_{i=1}^{\infty}\frac{n}{(ae^{i/a}+n)^2},\mbox{ and }V_0(\|W(a)\|^2_2)=2\sum\limits_{i=1}^{\infty}\frac{n^2}{(ae^{i/a}+n)^4}.$$
Furthermore note that by applying Lemma  \ref{lem: bound sum} (with $r=0$ and $l=2$) we get
\begin{align*}
\frac{a}{n}\log\big(\frac{n}{a}\big)\lesssim  I_a\frac{n}{(ae^{I_a/a}+n)^2}\leq\sum\limits_{i=1}^{I_a}\frac{n}{(ae^{i/a}+n)^2}\leq\sum\limits_{i=1}^{\infty}\frac{n}{(ae^{i/a}+n)^2}\lesssim \frac{a}{n}\log\big(\frac{n}{a}\big),
\end{align*}
while by applying the same lemma (with $r=0$ and $l=4$) the variance is bounded above by a multiple of $an^{-2}\log(n/a)$. Then similar reasoning to the previous proof results in that
\begin{align}
\inf_{f_0\in\ell_2(M)} \big(\sup\limits_{\underline{a}_n\leq a \leq\overline{a}_n}\|W(a)\|_2^2\leq C_2 (\overline{a}_n/n)\log(n/\overline{a}_n)\big)\stackrel{P_0}{\rightarrow}1.\label{eq: ub:variance:help}
\end{align}
Then in view of Lemma \ref{th: interval a polish tail}, the right hand side of the inequality in the preceding probability statement is further bounded from above by constant times $(\underline{a}_n/n)\log(n/\underline{a}_n)\log^2 n$.

Proof of \eqref{eq: bound bias}: First note that
$$\|B(a,f_0)\|^2_2 \leq \sum\limits_{i=1}^{I_a}n^{-2}a^2e^{2i/a}f^2_{0,i}+ \sum\limits_{i=I_a}^{\infty}f^2_{0,i}.$$
To bound the first term on the right hand side, we use the inequalities $a/n\leq \log(n/a)$ for $a\leq A_n$ and $\sum_{i=1}^{\infty}f_{0,i}^2<\infty$, and furthermore note the function $i\mapsto e^{i/a}/(i-a)$ is monotone increasing on the interval $[2a,I_a]$ hence it takes its maximum at $I_a$. Therefore the first part of the bias is bounded by
\begin{align*}
\sup\limits_{\underline{a}_n\leq a \leq\overline{a}_n}\sum\limits_{i=1}^{I_a}\frac{a^2e^{2i/a}f^2_{0,i}}{n^2}&
\leq\sup\limits_{\underline{a}_n\leq a \leq\overline{a}_n}\sum\limits_{i=1}^{2a}\frac{a^2e^{2i/a}f^2_{0,i}}{n^2}+\sup\limits_{\underline{a}_n\leq a \leq\overline{a}_n}\frac{a}{n}\frac{\log^2(\frac{n}{a})}{(\log(\frac{n}{a})-1)}g_n(a,f_0)\\
&\lesssim\frac{\overline{a}_n^2}{n^2}\sum\limits_{i=1}^{2a}f^2_{0,i}+\frac{\overline{a}_n}{n}\log (n)\log(\frac{n}{\overline{a}_n})\lesssim\frac{\overline{a}_n}{n}\log (n)\log(\frac{n}{\overline{a}_n}),
\end{align*}
using the definition of $\underline{a}_n$.  Furthermore in view of the polished tail assumption we have
$$\sum_{i=I_{\underline{a}_n}}^{\infty} f^2_{0,i}\leq L_0\sum_{i=I_{\underline{a}_n}}^{\rho I_{\underline{a}_n}} f^2_{0,i}\leq   \log(\frac{n}{\underline{a}_n})\sum\limits_{i=I_{\tilde{a}_n}}^{ I_{\tilde{a}_n}+\rho \tilde{a}_n} f^2_{0,i} ,$$
for some $\tilde{a}_n\in [\underline{a}_n,\rho\underline{a}_n]$. Therefore
\begin{align*}
\sum\limits_{i=I_{\underline{a}_n}}^{\infty}f_{0,i}^2&\lesssim\log(\frac{n}{\underline{a}_n})\sum\limits_{i=I_{\tilde{a}_n}}^{I_{\tilde{a}_n}+\rho \tilde{a}_n}\frac{n^2(i-\tilde{a}_n)e^{i/\tilde{a}_n}}{\tilde{a}_n\log^2(\frac{n}{\tilde{a}_n})(\tilde{a}_ne^{i/\tilde{a}_n}+n)^2}f^2_{0,i}\frac{\tilde{a}_n}{n}\log(\frac{n}{\tilde{a}_n}),\\
&\leq  \log(\frac{n}{\underline{a}_n}) g_n(\tilde{a}_n,f_0)\frac{\tilde{a}_n}{n}\log(\frac{n}{\tilde{a}_n})
\lesssim  \log^2(\frac{n}{\underline{a}_n})\log (n)\frac{\underline{a}_n}{n}.
\end{align*}
Combining the two bounds, we see that (\ref{eq: bound bias}) holds.

\begin{lemma}\label{lem: bound u}~There exists a $K>0$ such that for any $1<a_1 < a_2$ 
\begin{align}
V\Big(\frac{U_n(a_1)}{a_1\log(n/a_1)}-\frac{U_n(a_2)}{a_2\log(n/a_2)}\Big)\leq K(a_1-a_2)^2\frac{\log(n/a_1)}{a_1^3n^2}.\label{eq: UB:metricU}
\end{align}
\end{lemma}

\begin{proof}~
First note that
\begin{align}\label{eq: trivial bound u}
V\Big(\frac{U_n(a_1)}{a_1\log(n/a_1)}-\frac{U_n(a_2)}{a_2\log(n/a_2)}\Big)=2\sum\limits_{i=1}^{\infty}(\phi_i(a_1)-\phi_i(a_2))^2
\end{align}
with $\phi_i(a):=\frac{1}{a\log(n/a)(ae^{i/a}+n)}$. The derivative of $\phi_i(a)$ is given as $\phi'_i(a)=\phi_i(a)\Big(\frac{2(i-a)e^{i/a}}{a(ae^{i/a}+n)}+\frac{1}{a\log(n/a)}-\frac{1}{a}\Big)$, so we can see that $|\phi'_i(a)|\lesssim\big(\frac{(i+a)e^{i/a}}{a(ae^{i/a}+n)}\vee \frac{1}{a}\big)\phi_i(a)$. Thus in view of Lemma \ref{Lem: Diff} the right hand side of  \eqref{eq: trivial bound u} is bounded  by a multiple of
$$(a_1-a_2)^2\sup\limits_{a\in[a_1,a_2]}\sum\limits_{i=1}^{\infty}\Big(\frac{(i^2+a^2)e^{2i/a}}{a^2(ae^{i/a}+n)^2}\vee \frac{1}{a^2}\Big)\phi_i(a)^2.$$
Then in view of Lemma \ref{lem: bound term} (first with $m=2$ and then with $m=0$) and Lemma \ref{lem: bound sum} (first with $r=1$ and $l=2$ and second with $r=0$ and $l=2$) the preceding display is further bounded by the right hand side of \eqref{eq: UB:metricU}, finishing the proof of the statement.
\end{proof}

\section{Proof of Theorem \ref{thm: analytic coverage}}\label{sec: analytic coverage}
We use the notations introduced in Section \ref{sec: polish tail coverage}.

First recall that $f_0\in \hat{C}_n(L)$ if and only if $\|f_0-\hat{f}\|_2\leq L r_{\alpha}(\hat{a}_n)$. We show below that
\begin{align}
\label{eq: bound variance weak}
\inf\limits_{f_0\in A^{\gamma}(M)}P_{0}(\sup\limits_{\underline{a}_n\leq a \leq\overline{a}_n}\|W(a)\|_2^2\leq C_2\frac{\underline{a}_n}{n}\log\big(\frac{n}{\underline{a}_n}\big))\to 1,\\
\label{eq: bound bias weak}
\sup\limits_{\underline{a}_n\leq a \leq\underline{a}_n}\|B(a,f_0)\|_2^2\leq C_3\frac{\underline{a}_n}{n}\log\big(\frac{n}{\underline{a}_n}\big),
\end{align}
which together with \eqref{eq: bound radius} and Theorem \ref{th: consistency a} results in the statement.

The proof of assertion \eqref{eq: bound variance weak} follows by combining \eqref{eq: ub:variance:help} and the second inequality of Proposition \ref{prop: bound a}. Next note that similarly to the proof of (\ref{eq: bound bias}), we get that
\begin{align*}
\|B(a,f_0)\|_2^2\leq\sum_{i=1}^{I_a}\frac{a^2e^{2i/a}f^2_{0,i}}{n^2}+\sum_{i=I_{a}}^{\infty}f_{0,i}^2\lesssim \frac{\overline{a}_n}{n}\log\big(\frac{n}{\overline{a}_n}\big)+\sum_{i=I_{\underline{a}_n}}^{\infty}f_{0,i}^2.
\end{align*}
Furthermore
\begin{align*}
\sum\limits_{i=I_{\underline{a}_n}}^{\infty}f^2_{0,i}&=\sum\limits_{i=I_{\underline{a}_n}}^{\infty}e^{-2i\gamma}e^{2i\gamma}f^2_{0,i}
\lesssim e^{-2I_{\underline{a}_n}\gamma}
\lesssim \big(\frac{\underline{a}_n}{n}\big)^{2\underline{a}_n\gamma}\lesssim \frac{\underline{a}_n}{n} \log\big(\frac{n}{\underline{a}_n}\big)
\end{align*}
for $\gamma\geq 1/2$, finishing the proof of \eqref{eq: bound bias weak} and concluding the proof of the theorem.

\section{Proof of Theorem \ref{thm: self similar c-e} and the empirical Bayes part of Corollary \ref{cor: size}}\label{sec: self similar c-e}
In the proof we use again the notations introduced in Section \ref{sec: polish tail coverage}.

First note that $f_0\in\hat{C}_n(L_n)$ implies that $\|B(\hat{a}_n,f_0)\|_2\leq L_nr_{\alpha}(\hat{a}_n)+\|W(\hat{a}_n)\|_2$, which combined with Theorem \ref{th: consistency a} provides the upper bound
\begin{align}\label{eq: self-similar c-e}
P_{0}(f_0\in\hat{C}_n(L_n))\leq P_{0}\Big(\inf\limits_{a\leq\overline{a}_n}\|B(a,f_0)\|_2\leq L_n\sup\limits_{a\leq\overline{a}_n} r_{\alpha}(a)+\sup\limits_{a\leq\overline{a}_n}\|W(a)\|_2\Big)+o(1).
\end{align}\\
The proof of assertion (\ref{eq: bound radius}) also shows that there exists constants $C_1>0$ such that
\begin{align*}
\sup\limits_{f_0\in\Theta^{\beta}(M)}\sup\limits_{a \leq\overline{a}_n}r^2_{\alpha}(a)\leq C_1\frac{\overline{a}_n}{n}\log(\frac{n}{\overline{a}_n}).
\end{align*}
Then in view of assertion \eqref{eq: ub:variance:help} and Proposition \ref{prop: bound a}, both the squared radius $r_{\alpha}(a)^2$ and the variance term $\|W(a)\|_2^2$ are bounded by a constant multiplier of $n^{-2\beta/(1+2\beta)}(\log n)^{-1/(1+2\beta)}$.

Since for $f_0\in\Theta^\beta_s(m,M)$ we have $\|B(a,f_0)\|_2^2=\sum\limits_{i=1}^{\infty}\frac{a^2e^{2i/a}f_{0,i}^2}{(ae^{i/a}+n)^2}$ the bias is bounded from below by
\begin{align*}
\|B(a,f_0)\|^2\geq m\sum\limits_{i=I_a}^{\infty}i^{-1-2\beta}\gtrsim m I_a^{-2\beta} \gtrsim m a^{-2\beta}\log^{-2\beta}(n/a).
\end{align*}
As the function $a\mapsto a^{-2\beta}\log^{-2\beta}\big(n/a\big)$ is monotone decreasing for $a\leq A_n$, we see that $\inf\limits_{a\leq\overline{a}_n}\|B(a,f_0)\|_2^2\gtrsim m\overline{a}_n^{-2\beta}\log^{-2\beta}(n/\overline{a}_n)$. Hence in view of Proposition \ref{prop: bound a} the bias is bounded from below by $n^{-2\beta/(1+2\beta)}\log^{2\beta/(1+2\beta)}n$. Thus, the above inequalities imply that for arbitrary $f_0\in\Theta^\beta_s(m,M)$ the right hand side of \eqref{eq: self-similar c-e} is further bounded by
\begin{align*}
\sup_{f_0\in \ell_2(M)}P_0\big(n^{-\beta/(1+2\beta)}(\log n)^{\beta/(1+2\beta)}\leq L_nC n^{-\beta/(1+2\beta)}(\log n)^{-(1/2)/(1+2\beta))}\big)+o(1)=o(1),
\end{align*}
for arbitrary $L_n=o(\sqrt{\log n})$ concluding the proof of the theorem.

\section{Proof of Theorem \ref{th: consistency a}}\label{sec: consistency a}
First note that the derivative of the marginal likelihood function $\ell_n(a)$ is
\begin{align}\label{eq: deriv loglik}
\mathbb{M}_n(a)=\frac{1}{2}\Big(\sum\limits_{i=1}^{\infty}\frac{n^2Y_i^2e^{i/a}(i-a)}{a(ae^{i/a}+n)^2}-\sum\limits_{i=1}^{\infty}\frac{n(i-a)}{a^2(ae^{i/a}+n)}\Big),
\end{align}
with expected value
\begin{align}\label{eq: expect deriv loglik}
E_0[\mathbb{M}_n(a)]=\frac{1}{2}\Big(\sum\limits_{i=1}^{\infty}\frac{n^2(i-a)e^{i/a}f_{0,i}^2}{a(ae^{i/a}+n)^2}-\sum\limits_{i=1}^{\infty}\frac{n^2(i-a)}{a^2(ae^{i/a}+n)^2}\Big).
\end{align}
In the following subsections we show with the help of the score function $\mathbb{M}_n(a)$ that the marginal likelihood function $\ell_n(a)$ with probability tending to one has its global maximum outside of the set $[1,\underline{a}_n)\cup (\overline{a}_n,A_n]$.

\subsection{$\mathbb{M}_n(a)$ on $[1,\underline{a}_n)$}\label{sec: mmle:lb}
In this subsection we derive that the process $\mathbb{M}_n(a)$ is bounded from below by $-C\log^2(n/a)$ on $[1,\underline{a}_n]$, for some $C>0$, and is bigger than $\tilde{c}_0B \log^3(n/\underline{a}_n)$, for some  $0<\tilde{c}_0<\exp\{-2(4+1/\beta_0)\}/2$, on the interval 
\begin{align}
 \mathcal{I}_n\equiv\Big[\frac{\log(n/\underline{a}_n)}{1+\log(n/\underline{a}_n)}\underline{a}_n,\underline{a}_n \Big]\label{def: interval:lb}
\end{align}
 with probability going to one, where $B$ is the parameter in the definition of $\underline{a}_n$. Hence with probability tending to one for every $a\in[1,\underline{a}_n]/\mathcal{I}_n$
\begin{align*}
\ell_n(\underline{a}_n)-\ell_n(a)&\geq \int_{a}^{ \frac{\log(n/\underline{a}_n)}{1+\log(n/\underline{a}_n)}\underline{a}_n}\mathbb{M}_n(\tilde{a})d\tilde{a}+\int_{\mathcal{I}_n}\mathbb{M}_n(\tilde{a})d\tilde{a}\\
&\geq -(\underline{a}_n-a)C \log^2(n/\underline{a}_n)+\frac{\tilde{c}_0 B\underline{a}_n\log^3(n/\underline{a}_n)}{ \log(n/\underline{a}_n)}\\
&\geq (B\tilde{c}_0/2) \underline{a}_n\log^2(n/\underline{a}_n),
\end{align*}
for $B>2\tilde{c}_0^{-1}C$. Therefore the global maximum of $\ell_n(a)$ lies outside of the interval $[1,\underline{a}_n)$ with probability tending to one. It remained to show the stated lower bounds for $\mathbb{M}_n(a)$.

By leaving the non-negative stochastic part out we get the lower bound
\begin{align}
\mathbb{M}_n(a)\geq\frac{1}{2}\Big(\sum\limits_{i=1}^{a}\frac{n^2(i-a)e^{i/a}Y_i^2}{a(ae^{i/a}+n)^2}-\sum\limits_{i=1}^{\infty}\frac{n(i-a)}{a^2(ae^{i/a}+n)}\Big).\label{eq:help:LB}
\end{align}
In view of Lemma \ref{lem: bound determ} the deterministic part in \eqref{eq:help:LB} is bounded from below by a negative constant times $\log^2\big(n/a\big)$. The stochastic part is bounded from below by $-C\sum\limits_{i=1}^{a}Y_i^2$ and since $E_0[\sum\limits_{i=1}^{a}Y_i^2]=\sum\limits_{i=1}^{a}f_{0,i}^2+an^{-1}$ and $V_0\big(\sum\limits_{i=1}^{a}Y_i^2\big)=2n^{-1}\sum\limits_{i=1}^{a}f_{0,i}^2+an^{-2}\to 0$ for all $a\leq A_n$ it follows from Chebyshev's inequality that the sum $\sum\limits_{i=1}^{a}Y_i^2$ is bounded with probability going to 1, for all $f_0\in\ell_2(M)$.

Next we deal with the lower bound on the interval $a\in \mathcal{I}_n$. First note that $Y_i^2\geq f_{0,i}^2+2f_{0,i}Z_i/\sqrt{n}$ implying
\begin{align*}\mathbb{M}_n(a)\geq\frac{1}{2}\Big(\sum\limits_{i=1}^{a}\frac{n^2(i-a)e^{i/a}Y_i^2}{a(ae^{i/a}+n)^2}+\log^2(\frac{n}{a})g_n(a,f_0)+\mathbb{H}_n(a)-\sum\limits_{i=1}^{\infty}\frac{n(i-a)}{a^2(ae^{i/a}+n)}\Big),
\end{align*}
 with the centered Gaussian process 
\begin{align}
\mathbb{H}_n(a)=\sum\limits_{i=2a}^{\infty}\frac{n^{3/2}(i-a)e^{i/a}f_{0,i}Z_i}{a(ae^{i/a}+n)^2}.
\end{align}
Note that
\begin{align*}
V_0\Big(\frac{\mathbb{H}_n(a)}{\log^2(n/a)}\Big)&=\frac{1}{\log^4(n/a)}\sum\limits_{i=2a}^{\infty}\frac{n^3(i-a)^2e^{2i/a}f_{0,i}^2}{a^2(ae^{i/a}+n)^4}V_0(Z_i)\\
&\leq \frac{ng_n(a,f_0)}{a\log^2(n/a)}\max\limits_{i\geq 2a}\frac{(i-a)e^{i/a}}{(ae^{i/a}+n)^2}
\asymp\frac{g_n(a,f_0)}{a\log(n/a)},
\end{align*}
hence the diameter of the interval $ \mathcal{I}_n$ with respect to the metric
$$d_n^2(a_1,a_2)=V_0\Big( \frac{\mathbb{H}_n(a_1)}{\log^2(n/a_1)}-\frac{\mathbb{H}_n(a_2)}{\log^2(n/a_2)} \Big)$$
is bounded by a multiple of $\sup_{a\in\mathcal{I}_n}g_n(a,f_0)^{1/2}(a\log(n/a))^{-1/2}$. 

Next we give an upper bound for the covering number of the interval $\mathcal{I}_n$. Let us take $\epsilon$-balls centered at $a\in \mathcal{I}_n$, with $2a\in\mathbb{N}$. To cover the remaining part of the interval $ \mathcal{I}_n$ it is sufficient to cover all intervals of the form $(a,a+1/2)$, $2a\in\mathbb{N}\cap  2\mathcal{I}_n$. Note that on these intervals for every $a_1,a_2\in (a,a+1/2)$ we have $\lfloor 2a_1\rfloor-\lfloor 2a_2\rfloor=0$. Hence in view of  Lemma \ref{lem: bound h} we have $d_n(a_1,a_2)\lesssim |a_1-a_2| \sup_{a\in\mathcal{I}_n} \sqrt{\log(n/a)g_n(a,f_0)/a^3}$. Thus the covering number of the interval $(a,a+1/2)$ relative to $d_n$ is bounded from above by a multiple of $\varepsilon^{-1}\sup_{a\in\mathcal{I}_n} \sqrt{\log(n/a)g_n(a,f_0)/a^3}$, which implies that the covering number of the whole interval $ \mathcal{I}_n$ is bounded from above by constant times $\varepsilon^{-1}\sup_{a\in\mathcal{I}_n} \sqrt{\log^{-1}(n/a)g_n(a,f_0)/a}+\underline{a}_n/\log (n/\underline{a}_n)$. 

We show below that for any $c_0>2$ 
\begin{align}
e^{-2c_0}B\log n+o(1)\leq g_n(a,f_0)\leq e^{c_0}B\log n+o(1),\qquad \text{for $a\in\mathcal{I}_n$,}\label{eq: LB_hata_help}
\end{align}
hold. Therefore the covering number of $\mathcal{I}_n$ is bounded from above by a multiple of $\underline{a}_n+\varepsilon^{-1}\sqrt{\log^{-1}(n/\underline{a}_n)\log (n)/\underline{a}_n}$.

By Corollary 2.2.5 in \cite{vdv:wellner:96} (applied with $\psi(x)=e^{x^2}-1$) it follows that
\begin{align*}
E_0&\Bigg[\sup\limits_{a\in \mathcal{I}_n}\Bigg|\frac{\mathbb{H}_n(a)}{\log^2(n/a)}-\frac{\mathbb{H}_n(\underline{a}_n)}{\log^2(n/\underline{a}_n)}\Bigg|\Bigg]\\ 
&\qquad\qquad\qquad\lesssim \int_{0}^{C\log^{1/2}(n)I_{\underline{a}_n}^{-1/2} }\sqrt{\log\big(\underline{a}_n+\varepsilon^{-1}\sqrt{\log(n)\log^{-1}(n/\underline{a}_n)/\underline{a}_n}\big)}d\varepsilon\\ 
&\qquad\qquad\qquad\lesssim \int\limits_{0}^{C\log^{1/2}(n)I_{\underline{a}_n}^{-1/2} }\sqrt{\log \underline{a}_n}d\varepsilon+\int_0^1 \log(1/\varepsilon)d\varepsilon=O(1).
\end{align*}
Therefore the process $\mathbb{M}_n(a)$ can be bounded from below on $a\in\mathcal{I}_n$ by
\begin{align*}
\mathbb{M}_n(a)\geq 2^{-1}\inf\limits_{a\in \mathcal{I}_n}\Big\{\log^2(n/a)\Big(Be^{-2c_0}\log n-C\Big)\\+\sum\limits_{i=1}^{a}\frac{n^2(i-a)e^{i/a}Y_i^2}{a(ae^{i/a}+n)^2}-\sum\limits_{i=1}^{\infty}\frac{n(i-a)}{a^2(ae^{i/a}+n)}\Big\}
\end{align*}
with probability going to one. In view of \eqref{eq: LB_hata_help} and since the
third and fourth terms on the right hand side of the preceding display are bounded from below by a fixed negative constant, we get that with probability tending to one $\mathbb{M}_n(a)\geq \tilde{c}_0B\log n$, with $\tilde{c}_0=e^{-2c_0}/2$ .

It remained to verify assertion \eqref{eq: LB_hata_help}. First note that
\begin{align*}
\frac{n^2}{\log^2(n/\underline{a}_n)}\sum_{i=c_0 I_{\underline{a}_n}}^\infty\frac{(i-\underline{a}_n)e^{i/\underline{a}_n}}{\underline{a}_n(\underline{a}_ne^{i/\underline{a}_n}+n)^2}f_{0,i}^2&\leq\frac{n^2}{\underline{a}_n^3\log^2(n/\underline{a}_n)}\sum_{i=c_0 I_{\underline{a}_n}}^\infty i e^{-i/\underline{a}_n}f_{0,i}^2\\
&\lesssim\frac{A_n^{c_0-2}}{n^{c_0-2}\log(n/\underline{a}_n)}\|f_{0}\|_2^2
=o(1).
\end{align*}
Furthermore, in view of the inequality $c_0I_{\underline{a}_n}(a^{-1}-\underline{a}_n^{-1})\leq c_0$, for $a\in \mathcal{I}_n$,  we have that 
\begin{align*}
g_n(a,f_0)
&\geq \frac{n^2}{\log^2(n/a)}\sum_{i=2\underline{a}_n}^{c_0 I_{\underline{a}_n}}\frac{(i-a)e^{i/a}}{a(ae^{i/a}+n)^2}f_{0,i}^2\\
&\geq \frac{n^2}{e^{2c_0}\log^2(n/\underline{a}_n)}\sum_{i=2\underline{a}_n}^{c_0 I_{\underline{a}_n}}\frac{(i-\underline{a}_n)e^{i/\underline{a}_n}}{\underline{a}_n(\underline{a}_ne^{i/\underline{a}_n}+n)^2}f_{0,i}^2.
\end{align*}
By combining the preceding two displays we get that
\begin{align*}
g_n(a,f_0)&\geq e^{-2c_0}g_n(\underline{a}_n,f_0)-\frac{e^{-2c_0}n^2}{\log^2(n/\underline{a}_n)}\sum_{i=c_0 I_{\underline{a}_n}}^\infty\frac{(i-\underline{a}_n)e^{i/\underline{a}_n}}{\underline{a}_n(\underline{a}_ne^{i/\underline{a}_n}+n)^2}f_{0,i}^2\\
&\geq e^{-2c_0}B\log n+o(1),
\end{align*} 
finishing the proof of the first inequality in \eqref{eq: LB_hata_help}. The proof of the second inequality goes accordingly.

\begin{lemma}\label{lem: bound determ}
There exists a constant $K>0$ such that for any $a\in[1,n)$
$$\sum\limits_{i=1}^{\infty}\frac{n(i-a)}{a^2(ae^{i/a}+n)}\leq K \log^2(n/a)$$
\end{lemma}

\begin{proof}~
Note that
\begin{align*}
\sum\limits_{i=1}^{\infty}\frac{n(i-a)}{a^2(ae^{i/a}+n)}&\leq \sum\limits_{i=1}^{\infty}\frac{ni}{a^2(ae^{i/a}+n)}\leq\sum\limits_{i=1}^{I_a}\frac{i}{a^2}+\sum\limits_{i=I_a}^{\infty}\frac{nie^{-i/a}}{a^3}\\
&\lesssim\log^2(n/a)+\frac{\log(n/a)}{a}\lesssim\log^2(n/a).
\end{align*}
\end{proof}

\begin{lemma}\label{lem: bound h}~
There exists a constant $K>0$ such that for any $0<a_1 < a_2$, $\lfloor 2a_2\rfloor -\lfloor 2a_1\rfloor=0$ \\
\begin{align*}
V_0\Big(\frac{\mathbb{H}_n(a_1)}{\log(n/a_1)^2}-\frac{\mathbb{H}_n(a_2)}{\log(n/a_2)^2}\Big)\leq K(a_1-a_2)^2  \sup_{a\in[a_1,a_2]}\frac{\log(n/a)g_n(a,f_0)}{a^3}.\\
\end{align*}
\end{lemma}

\begin{proof}
Recall that the left hand side of the display in the lemma was denoted by $d_n^2(a_1,a_2)$ and note that
\begin{align}\label{eq: trivial bound h}
d_n^2(a_1,a_2)=\sum\limits_{i=2a_2}^{\infty}\big(\phi_i(a_1)-\phi_i(a_2)\big)^2n^3f_{0,i}^2
\end{align}
with $\phi_i(a):=\frac{(i-a)e^{i/a}}{\log(n/a)^2a(ae^{i/a}+n)^2}$. Then by elementary, but cumbersome computations we get that $|\phi'_i(a)|\lesssim ia^{-2}\phi_i(a)$. Thus, in view of Lemma \ref{Lem: Diff}, the right hand side of \eqref{eq: trivial bound h} is bounded by
\begin{align*}
n^3(a_1-a_2)^2&\sup\limits_{a\in[a_1,a_2]}\sum\limits_{i=2a}^{\infty}\frac{i^2}{a^4}\phi_i(a)^2f_{0,i}^2\\
&\lesssim (a_1-a_2)^2\sup_{a\in[a_1,a_2]}g_n(a)\sup_{i\in\mathbb{N}}\frac{ni^3e^{i/a}}{a^5\log^2(n/a)(ae^{i/a}+n)^2}.
\end{align*}
Then the statement of the lemma follows by applying Lemma \ref{lem: bound term} (with $m=3$).
\end{proof}

\subsection{$\mathbb{M}_n(a)$ on $[\overline{a}_n,A_n]$}
By assuming $\overline{a}_n>K_0$, we have $h_n(a,f_0)\leq b$ for $a\in[\overline{a}_n,A_n]$. Next we prove that for sufficiently large choice of $K_0>0$
\begin{align}\label{eq: bound expectation}
\limsup_n\sup\limits_{f_0\in \ell_2(M)}\sup\limits_{a\in[\overline{a}_n,A_n]}E_{0}\Big[\frac{\mathbb{M}_n(a)}{\log^2(n/a)}\Big]<-2^{-5},\\
\label{eq: bound randomness}
\limsup_n\sup\limits_{f_0\in \ell_2(M)}E_{0}\Bigg[\sup\limits_{a\in[\overline{a}_n,A_n]}\frac{|\mathbb{M}_n(a)-E_0[\mathbb{M}_n(a)]|}{\log^2(n/a)}\Bigg]\leq 2^{-6}.
\end{align}
These imply that with probability tending to one $\mathbb{M}_n(a)<-2^{-6}\log^2(n/a)$, for every $a\in[\bar{a}_n,A_n]$, hence the marginal likelihood function $\ell_n(a)$ is monotone decreasing and does not attain its global (or local) maximum on the interval $[\overline{a}_n,A_n]$, i.e.
\begin{align}
\inf_{f_0\in\ell_2(M)}P_0(\hat{a}_n\leq \overline{a}_n)\rightarrow 1.\label{eq: UB:hyperparameter}
\end{align}

Proof of assertion (\ref{eq: bound expectation}): For $a\geq\overline{a}_n$ the expectation satisfies
\begin{align*}
E_0\Big[\frac{\mathbb{M}_n(a)}{\log^2(n/a)}\Big]&=\frac{1}{2}\Big(h_n(a,f_0)-\frac{1}{\log^2(n/a)}\sum\limits_{i=1}^{\infty}\frac{n^2(i-a)}{a^2(ae^{i/a}+n)^2}\Big)\\
&\leq\frac{1}{2}\Big(b-\frac{1}{\log^2(n/a)}\sum\limits_{i=1}^{\infty}\frac{n^2(i-a)}{a^2(ae^{i/a}+n)^2}\Big).
\end{align*}
In view of Lemma \ref{lem: bound sum} (with $r=0$ and $l=2$), we have $\sum\limits_{i=1}^{\infty}\frac{n^2}{a(ae^{i/a}+n)^2}\lesssim \log(n/a)$. Furthermore,
\begin{align*}
\sum\limits_{i=1}^{\infty}\frac{in^2}{a^2(ae^{i/a}+n)^2}\geq \sum\limits_{i=1}^{I_a}\frac{i}{4a^2}=\frac{I_a(I_a+1)}{8a^2}\geq 2^{-3}\log^2\big(\frac{n}{a}\big),
\end{align*}
which implies that
\begin{align*}
E_0\big[\mathbb{M}_n(a)/\log^{2}(n/a)\big]\leq (b-2^{-3}+o(1))/2,
\end{align*}
concluding the proof of assertion \eqref{eq: bound expectation}, for small enough choice of $b$ ($b<2^{-4}$ is small enough).

Proof of assertion \eqref{eq: bound randomness}: In view of Corollary 2.2.5 in \cite{vdv:wellner:96} (applied with $\psi(x)=x^2$) it is sufficient to show that there exist universal constants $K_1,K_2>0$ such that for any $a\in [\overline{a}_n,A_n]$
\begin{align}\label{eq: bound var}
V_{0}\big(\mathbb{M}_n(a)/\log^{2}(n/a)\big)\leq K_1/\log (n/a),\\
\label{eq: bound entropy}
\int\limits_{0}^{diam_n}\sqrt{N(\varepsilon, [\overline{a}_n,A_n],d_n)}d\varepsilon\leq  K_2/K_0^{1/4},
\end{align}
where $d_n$ is the semimetric defined by $d_n^2(a_1,a_2):= V_{0}\Big(\frac{\mathbb{M}_n(a_1)}{\log^2(n/a_1)}-\frac{\mathbb{M}_n(a_2)}{\log^2(n/a_2)}\Big)$, $diam_n$ is the diameter of $ [\overline{a}_n,A_n]$ relative to $d_n$ and $N(\varepsilon,S,d_n)$ is the minimal number of $d_n$-balls of radius $\varepsilon$ needed to cover the set $S$, since by sufficiently large choice of $K_0$ ($K_0\geq (2^6K_2)^4$ is sufficiently large) assertion \eqref{eq: bound randomness} holds.

Note that Lemma \ref{lem: bound var} immediately implies assertion (\ref{eq: bound var}) and 
$$diam_n\lesssim  \sup_{a\in[\bar{a}_n,A_n]}(a \log(n/a))^{-1/2}\lesssim \log^{-1/2} n.$$
Then let us introduce the cover
\begin{align*}
[\overline{a}_n,A_n]\subset\bigcup\limits_{k=0}^{K_n-1}[ 2^k\overline{a}_n,2^{k+1}\overline{a}_n]
\end{align*}
with $K_n= \lceil\log(A_n/\overline{a}_n)\rceil$. In view of Lemma \ref{lem: bound entropy} on the interval $a_1,a_2\in[ 2^k\overline{a}_n,2^{k+1}\overline{a}_n]$ 
\begin{align*}
d_n(a_1,a_2)\lesssim  \big(2^k\overline{a}_n\big)^{-3/2}\log^{1/2}(n)|a_1-a_2|,
\end{align*}
hence 
\begin{align*}
N(\varepsilon,[\overline{a}_n,A_n],d_n)\lesssim 
\sum_{k=0}^{K_n-1} \frac{\log^{1/2}(n)}{\varepsilon \big(2^k \overline{a}_n\big)^{1/2}}\lesssim \frac{\log^{1/2}(n)}{\varepsilon\overline{a}_n^{1/2}}.
\end{align*}
This results in 
\begin{align*}
\int\limits_{0}^{diam_n}\sqrt{N(\varepsilon, [\overline{a}_n,A_n],d_n)}d\varepsilon\leq  K_2/\overline{a}_n^{1/4}\leq K_2/K_0^{1/4}.
\end{align*}

\begin{lemma}\label{lem: bound var}~For all $a\in[\overline{a}_n,A_n]$, we have $V_{0}\big(\mathbb{M}_n(a)/\log^{2}(n/a)\big)\lesssim (a\log (n/a))^{-1}$.
\end{lemma}

\begin{proof}~We know that the $Y_i$'s are independent and $V_0(Y_i^2) = 2/n^2+4f_{0,i}^2/n$, so the variance is equal to
\begin{align}
V_{0}\Big(\frac{\mathbb{M}_n(a)}{\log^2(n/a)}\Big)&=\frac{1}{4}\sum\limits_{i=1}^{\infty}\frac{n^4V_0(Y_i^2)e^{2i/a}(i-a)^2}{a^2\log^4(n/a)(ae^{i/a}+n)^4}\nonumber\\
&=\frac{1}{2}\sum\limits_{i=1}^{\infty}\frac{n^2e^{2i/a}(i-a)^2}{a^2\log^4(n/a)(ae^{i/a}+n)^4}+\sum\limits_{i=1}^{\infty}\frac{n^3e^{2i/a}(i-a)^2f_{0,i}^2}{a^2\log^4(n/a)(ae^{i/a}+n)^4}.\label{eq:help:Lem4}
\end{align}
In view of $(i-a)^2\leq a^2+i^2$, for any $a,i>0$, and by applying  Lemma  \ref{lem: bound term} (with $m=2$) and Lemma \ref{lem: bound sum} (first with $r=2$ and $l=4$ and then with $r=1$ and $l=2$) the first sum in $\eqref{eq:help:Lem4}$ is bounded from above by a multiple of
\begin{align*}
\sum_{i=1}^{\infty}\frac{n^2e^{2i/a}}{\log^4(n/a)(ae^{i/a}+n)^4}&+\sum_{i=1}^{\infty}\frac{ne^{i/a}}{a\log^2(n/a)(ae^{i/a}+n)^2}\\
&\lesssim \frac{1}{a\log^3(n/a)}+\frac{1}{a\log(n/a)}\lesssim \frac{1}{a\log (n/a)}.
\end{align*}
Similarly, following from Lemma  \ref{lem: bound term} (with $m=1$ and $m=-1$) and $h_n(a)\leq b$ for $a\geq\overline{a}_n$, the second sum in \eqref{eq:help:Lem4} is bounded by a multiple of
\begin{align*}
\Big(\max\limits_{i\in\mathbb{N}}&\frac{ane^{i/a}}{i\log^2(n/a)(ae^{i/a}+n)^2}+ \max\limits_{i\in\mathbb{N}}\frac{ine^{i/a}}{a\log^2(n/a)(ae^{i/a}+n)^2} \Big) h_n(a,f_0)\\
&\lesssim\Big(\frac{1}{\log^3(n/a)}+ \frac{1}{a\log(n/a)}\Big)\lesssim\frac{1}{\log (n/a)},
\end{align*}
concluding the proof of the lemma.
\end{proof}

\begin{lemma}\label{lem: bound entropy}~
For all $1\leq a_1<a_2<A_n$, we have 
\begin{align*}
d^2_n(a_1,a_2)\leq C_0(a_1-a_2)^2\sup\limits_{a\in[a_1,a_2]}\frac{\log(n/a)}{a^3}(1+h_n(a,f_0)),
\end{align*}
for some universal constant $C_0>0$.
\end{lemma}

\begin{proof}~Note that
\begin{align*}
d_n^2(a_1,a_2)=n^4\sum\limits_{i=1}^{\infty}(\phi_i(a_1)-\phi_i(a_2))^2V_0(Y_i^2),
\end{align*}
 with $\phi_i(a)=\frac{e^{i/a}(i-a)}{2a\log^2(n/a)(ae^{i/a}+n)^2}$. By elementary computations one can see that $|\phi_i(a)'|^2\lesssim (i^2 a^{-4}+a^{-2})\phi_i^2(a)$, hence in view of Lemma \ref{Lem: Diff},
\begin{align*}
d_n^2(a_1,a_2)\lesssim (a_1-a_2)^2n^4\sup\limits_{a\in[a_1,a_2]}\sum\limits_{i=1}^{\infty}\frac{e^{2i/a}(i^4+a^4)}{a^6\log^4(n/a)(ae^{i/a}+n)^4}V_0(Y_i^2).
\end{align*}
Since $V_0(Y_i^2) = 2/n^2+4f_{0,i}^2/n$ the preceding sum is bounded by 
\begin{align}
\sum\limits_{i=1}^{\infty}\frac{2e^{2i/a}(i^4+a^4)}{a^6n^2\log^4(n/a)(ae^{i/a}+n)^4}+\sum\limits_{i=1}^{\infty}\frac{4e^{2i/a}(i^4+a^4)}{a^6n\log^4(n/a)(ae^{i/a}+n)^4}f^2_{0,i}.\label{eq: help:Lem5}
\end{align}
Then in view of Lemma \ref{lem: bound term} (applied with $m=4$ and $m=0$) and Lemma \ref{lem: bound sum} (applied with $r=1$ and $l=2$) the first term of \eqref{eq: help:Lem5}
is bounded from above by a multiple of
\begin{align*}
\sum\limits_{i=1}^{\infty}&\frac{e^{i/a}}{a^3n^3(ae^{i/a}+n)^2}
\lesssim\frac{\log(n/a)}{a^3n^4}.
\end{align*}
Similarly in view of Lemma \ref{lem: bound term} (with $m=3$ and $m=-1$)  the second term of \eqref{eq: help:Lem5} is bounded by
\begin{align*}
\max\limits_{i\in\mathbb{N}}&\frac{\big((i/a)^3+(i/a)^{-1}\big)e^{i/a}}{a^2n^3\log^2(n/a)(ae^{i/a}+n)^2}
 h_n(a,f_0)\\
&\qquad\lesssim\Big(\frac{\log(n/a)}{a^3n^4}+\frac{1}{n^5 a^2}\Big)h_n(a,f_0)\lesssim \frac{\log (n/a)}{a^3n^4}h_n(a,f_0),
\end{align*}
concluding the proof of the lemma. 
\end{proof}

\section{Proof of Theorem \ref{thm: mod_eb}}\label{sec: mod_eb}
Similarly to the previous sections we use the notations introduced in Section \ref{sec: polish tail coverage}. We show below that there exists a constant $c>0$ depending only on $m,M$ and $\beta_0$ such that
\begin{align}
\inf_{\beta\geq \beta_0}\inf_{f_0\in \Theta_s^\beta(m,M)}P_0(\hat{a}_n\geq c (n/\log n)^{1/(1+2\beta)}/\log n )\rightarrow 1,\label{eq: lb_mod:mmle}
\end{align}
which combined with Proposition \ref{prop: bound a} and Theorem \ref{th: consistency a} results in
\begin{align*}
\inf_{\beta\geq \beta_0}\inf_{f_0\in \Theta_s^\beta(m,M)} P_0(c (n/\log n)^{1/(1+2\beta)}\leq \tilde{a}_n\leq C (n/\log n)^{1/(1+2\beta)})\rightarrow 1,
\end{align*}
for some positive constants $c,C$. Let us introduce then the notation 
$$\tilde{\mathcal{I}}_n=[c (n/\log n)^{\frac{1}{1+2\beta}},C (n/\log n)^{\frac{1}{1+2\beta}}].$$
 As before, note that $f_0\in \hat{C}_n(L)$ is equivalent to $\|f_0-\hat{f}\|_2\leq L r_{\alpha}(\tilde{a}_n)$, hence  
by proving that 
\begin{align*}
\inf_{a\in \tilde{\mathcal{I}}_n}r^2_{\alpha}(a)\geq C_1(n/\log n)^{-2\beta/(1+2\beta)},\\
\inf_{\beta\geq\beta_0}\inf_{f_0\in\Theta_s^\beta(m,M)}P_{0}\big(\inf_{a\in \tilde{\mathcal{I}}_n}\|W(a)\|_2^2\leq C_2(n/\log n)^{-2\beta/(1+2\beta)}\big)\to 1,\\
\sup_{\beta\geq\beta_0}\sup_{f_0\in\Theta_s^\beta(m,M)}\sup_{a\in \tilde{\mathcal{I}}_n}\|B(a,f_0)\|_2^2\leq C_3(n/\log n)^{-2\beta/(1+2\beta)},
\end{align*}
hold for some constants $C_1,C_2,C_3>0$, the statement of the theorem follows immediately. The proof of the first two inequalities follow from \eqref{eq: bound radius} and \eqref{eq: ub:variance:help} (with $\underline{a}_n$ and $\overline{a}_n$ replaced by  a multiple of $(n/\log n)^{1/(1+2\beta)}$), respectively. To prove the last inequality we note that for $f_0\in\Theta_s^\beta(m,M)$, $a\in \tilde{\mathcal{I}}_n$, and $\beta\geq\beta_0$ we have that 
\begin{align*}
\|B(a,f_0)\|^2_2&\lesssim \sum_{i=1}^{I_a/2}a^2e^{2i/a}n^{-2}i^{-1-2\beta}+\sum_{i=I_a/2}^{\infty}i^{-1-2\beta}\lesssim a/n + I_a^{-2\beta}\\
& = o\big( (n/\log n)^{-2\beta/(1+2\beta)}\big).
\end{align*}

It remained to prove assertion \eqref{eq: lb_mod:mmle}. Let us introduce the slightly modified version of $\underline{a}_n$ as
\begin{align*}
\underline{a}_n':=\sup\{a\in[1,A_n]: g_n(a,f_0)\geq B\},
\end{align*}
for some sufficiently large constant $B>0$ to be specified later. Then we show below that
\begin{align}
P_0(\hat{a}_n\geq \underline{a}_n')\rightarrow 1,\quad\text{and}\quad \underline{a}_n'\geq c (n/\log n)^{1/(1+2\beta)}/\log n,\label{eq: help:mod_eb}
\end{align}
for some sufficiently small constant $c>0$.

For the second statement note that
\begin{align}
g_n(a,f_0)\geq \frac{m}{\log^2(n/a)}n^2\sum_{i=I_a}^{\infty}  e^{-i/a}i^{-2\beta}
\gtrsim mn a^{-1-2\beta}\log^{-2-2\beta} (n/a),\label{eq: mod_eb:help2}
\end{align}
hence for any fixed $B>0$ there exists a small enough $c>0$ such that the right hand side of the preceding display with $a=c (n/\log n)^{1/(1+2\beta)}/\log n$ is bigger than $B$. It remained to deal with the first part of \eqref{eq: help:mod_eb}. We show below that with probability tending to one $\inf_{a\in [\underline{a}_n'/2,\underline{a}_n']}\mathbb{M}_n(a)\geq cB\log^2(n/a)$, for some small enough constant $c>0$, not depending on $B$. Then with probability tending to one for any $a\in[1,\underline{a}_n'/2]$ we have
\begin{align*}
\ell_n(\underline{a}_n)-\ell_n(a)&\geq \int_{a}^{\underline{a}_n'/2}\mathbb{M}_n(\tilde{a})d\tilde{a}+\int_{  [\underline{a}_n'/2,\underline{a}_n']}\mathbb{M}_n(\tilde{a})d\tilde{a}\\
&\geq -(\underline{a}_n'/2-a)C \log^2(n/\underline{a}_n)+c B(\underline{a}_n'/2)\log^2(n/\underline{a}_n')\\
&\geq (c/4) B\underline{a}_n'\log^2(n/\underline{a}_n'),
\end{align*}
for large enough choice of $B>0$, hence the global maximum of $\ell_n(a)$ lies outside of the interval $[1,\underline{a}_n']$.

 It remained to verify the lower bound for $M_n(a)$. First note that for $a\leq A_n=o(n)$
\begin{align*}
g_n(a,f_0)&\leq \frac{M}{\log^2(n/a)}\Big(\frac{1}{a}\sum_{i=2a}^{I_a}e^{i/a}i^{-2\beta}+\frac{n^2}{a^3}\sum_{i=I_a}^{\infty} e^{-i/a}i^{-2\beta}\Big)\\
&\leq c_{M,\beta} n a^{-1-2\beta}(\log n)^{-2-2\beta},
\end{align*}
hence $\underline{a}_n'\leq c_{M,\beta}' B^{-1/(1+2\beta)}(n/\log n)^{1/(1+2\beta)}/\log n$. Therefore in view of \eqref{eq: mod_eb:help2} for every $a\geq \underline{a}_n'/2$ we have $g_n(a,f_0)\geq c_{M,\beta,m}B$, for some positive constant $c_{M,\beta,m}>0$ not depending on $B$. Similarly we can show that $g_n(a,f_0)\leq c_{M,\beta,m}'B$, for every $a\geq \underline{a}_n'/2$, for some $c_{M,\beta,m}'>0$ not depending on $B$. Then following the same line of reasoning as in 
Section \ref{sec: mmle:lb}, with the only main difference that instead of the interval given in \eqref{def: interval:lb} we are working with the interval $[\underline{a}_n'/2,\underline{a}_n']$ we get that with porbability going to one
\begin{align*}
\inf\limits_{a\in [\underline{a}_n'/2,\underline{a}_n']}\mathbb{M}_n(a)&\geq 2^{-1}\inf\limits_{a\in [\underline{a}_n'/2,\underline{a}_n']}\Big\{\log^2(n/a)\Big(c_{M,\beta,m}B  -\sqrt{c_{M,\beta,m}' B}\Big)\\&\quad+\sum\limits_{i=1}^{a}\frac{n^2(i-a)e^{i/a}Y_i^2}{a(ae^{i/a}+n)^2}-\sum\limits_{i=1}^{\infty}\frac{n(i-a)}{a^2(ae^{i/a}+n)}\Big\}\\
&\gtrsim B\log^2(n/\underline{a}_n'),
\end{align*}
for large enough choice of $M>0$, finishing the proof of the theorem.

\section{Proofs for the Hierarchical Bayes procedure}\label{sec: hb}
In this section we prove the results on the hierarchical Bayes procedure (i.e. Theorems \ref{thm: hb:contraction} and \ref{thm: polish tail coverage} and Corollary \ref{cor: size}) based on the results derived for the empirical Bayes procedure. First we state that under the conditions of Theorem \ref{thm: hb:contraction} the hyper-posterior distribution on the hyper-parameter $a$ concentrates most of its mass on the interval $\mathcal{I}_n=[\underline{a}_n\log(n)/(1+\log n) ,C\overline{a}_n]$, for some large enough constant $C>0$. 
\begin{lemma}\label{lem: HB}~If $a\sim\pi(.)$ such that $\pi$ verifies Assumption \ref{assump: HB} then for sufficiently large $C>0$ we have for every $\beta_0>0$ that
\begin{align*}
\inf_{\beta>\beta_0}\inf_{f_0\in\Theta^{\beta}(M)}E_0\Pi\Big(\underline{a}_n\log(n)/(1+\log n) \leq a\leq C\overline{a}_n|Y\Big)= 1+o(1/n).
\end{align*}
\end{lemma}

\subsection{Proof of Theorem \ref{thm: hb:contraction}}\label{sec: hb:contraction}
Take $\varepsilon_n=(n/\log^2n)^{-\beta/(1+2\beta)}$. Then following from Lemma \ref{lem: HB}, we have
\begin{align*}
\sup_{f_0\in\Theta^\beta(M)}&E_0\Pi\big(f:\, \|f-f_0\|_2>M_n\varepsilon_n |Y\big)\leq\sup_{f_0\in\Theta^\beta(M)}\Big(E_0\Pi(a\notin \mathcal{I}_n|Y)\\
 &\qquad\qquad+E_0\sup_{a\in \mathcal{I}_n}\Pi_a\big(f:\, \|f-f_0\|_2>M_n\varepsilon_n |Y\big)\Big)=o(1),
\end{align*}
where the last equation follows by similar arguments as given in \eqref{eq : contraction mark ineq} and the displays below it (the only difference is that the supremum is taken over the interval $\mathcal{I}_n$ instead of $[\underline{a}_n,\overline{a}_n]$, but it only changes the constant factors which do not play an essential role. This concludes the proof of the theorem.

\subsection{Proof of Theorem \ref{thm: polish tail coverage} - Hierarchical Bayes part}\label{sec: hb:Thm2.4}

Let us introduce the notations $W=\hat{f}-E_0\hat{f}$ and $B(f_0)=E_0\hat{f}-f_0$, for the centered hierarchical posterior mean and the bias of the posterior mean, respectively. Then $P_{0}(f_0\in \hat{C}(L\log n))$ if and only if
\begin{align}
\|W\|_2\leq L\log (n)r_{\alpha}-\|B(f_0)\|_2\label{eq: help:hb1}
\end{align}
holds. Using assertions \eqref{eq: bound radius}, \eqref{eq: bound variance}, and \eqref{eq: bound bias} we show below that, there exist constants $\tilde{C}_1,\tilde{C}_2,\tilde{C}_3>0$, such that 
\begin{align}\label{eq: bound radius:hb}
r^2_{\alpha}\geq \tilde{C}_1(\underline{a}_n/n)\log (n/\underline{a}_n),\\
\label{eq: bound variance:hb}
\inf\limits_{f_0\in\Theta_{pt}(L_0,N_0,\rho)}P_{0}\big(\|W\|^2_2\leq  \tilde{C}_2 (\underline{a}_n/n)\log(n/\underline{a}_n)\log^2 n\big)\to 1,\\
\label{eq: bound bias:hb}
\|B(f_0)\|^2_2\leq  \tilde{C}_3(\underline{a}_n/n)\log^2\big(n/\underline{a}_n)\log n,
\end{align}
resulting in \eqref{eq: help:hb1} for sufficiently large choice of $L>0$.

Proof of \eqref{eq: bound radius:hb}: Let us take any $\alpha'>\alpha$ and note that in view of \eqref{eq: bound radius} we have $$\inf_{a\in \mathcal{I}_n}r_{\alpha'}(a)^2\geq C_1 (\underline{a}_n/n)\log(n/\underline{a}_n).$$ Next, in view of Lemma \ref{lem: HB} and Anderson's lemma, we get for arbitrary $r\leq \inf_{a\in \mathcal{I}_n}r_{\alpha'}(a)$ that
\begin{align*}
\Pi(f:\, \|f-\hat{f}\|_2\leq r|Y)
&=\int_{\mathcal{I}_n}\Pi_a(f:\, \|f-\hat{f}\|_2\leq r|Y)\pi(a|Y)da+o(1)\\
&\leq \int_{\mathcal{I}_n}\Pi_a(f:\, \|f-\hat{f}_a\|_2\leq r_{\alpha'}(a)|Y)\pi(a|Y)da+o(1)\\
&\leq 1-\alpha'+o(1),
\end{align*}
hence $r_\alpha^2\geq\inf_{a\in \mathcal{I}_n}r_{\alpha'}(a)^2\geq C_1  (\underline{a}_n/n)\log(n/\underline{a}_n)$.

Proof of \eqref{eq: bound variance:hb}: Note that by triangle inequality, Fubini's theorem, assertion \eqref{eq: bound variance}, and Lemma \ref{lem: HB} we get that under the polished tail condition with $P_0$-probability tending to one
\begin{align*}
\|W\|_2&=\Big\|\int (\hat{f}_a-E_0\hat{f}_a)\pi(a|Y)da \Big\|_2\\
& \leq\sup_{a\in \mathcal{I}_n}\|W(a)\|_2\pi(\mathcal{I}_n|Y)+\sup_{1\leq a\leq A_n}\|W(a)\|_2\pi(\mathcal{I}_n^c|Y)\\
&\leq  (C_2\underline{a}_n/n)^{1/2}\log(n/\underline{a}_n)^{1/2}\log n+o(1/n)
\end{align*}
where $\pi(\mathcal{I}_n|Y)$ denotes (by slightly abusing our notation) the posterior probability that the hyper-parameter $a$ lies in the interval $\mathcal{I}_n$ and in the last inequality we used  in view of the proof of assertion \eqref{eq: bound variance} that $\sup_{1\leq a\leq A_n}\|W(a)\|_2=O(1)$.

Proof of \eqref{eq: bound bias:hb}: Similarly to the proof of \eqref{eq: bound variance:hb} we get that
\begin{align*}
\|B(f_0)\|_2^2&\lesssim \sup_{a\in \mathcal{I}_n}\|B(a,f_0)\|_2^2+o(\sup_{a\in [1,A_n]}\|B(a,f_0)\|_2^2 /n)\\
&\leq C_3(\underline{a}_n/n)\log^2\big(n/\underline{a}_n)\log n+o(1/n),
\end{align*}
where the last inequality follows from $\|B(a,f_0)\|_2^2\leq \|f_0\|_2^2=O(1)$, finishing the proof of the theorem.

\subsection{Proof of Corollary \ref{cor: size}}
Let $\varepsilon_n=(n/\log^2 n)^{-\beta/(1+2\beta)}$ and first note that in view of assertions \eqref{eq: bound variance:hb} and \eqref{eq: bound bias:hb} combined with triangle inequality and Proposition \ref{prop: bound a} we have with $P_0$-probability tending to one that
\begin{align*}
\|f_0-\hat{f}\|_2\leq \|W\|_2+\|B(f_0)\|_2\lesssim \sqrt{\underline{a}_n/n}\log(n/\underline{a}_n)\lesssim \varepsilon_n.
\end{align*}
Then in view of Theorem \ref{thm: hb:contraction} and by applying again the triangle inequality we get with probability tending to one that
\begin{align*}
\Pi(f:\, \|f-\hat{f}\|_2\leq M_n\varepsilon_n|Y)\geq \Pi(f:\, \|f-f_0\|_2\leq M_n\varepsilon_n-\|f_0-\hat{f}\|_2|Y)=1-o(1),
\end{align*}
concluding the proof of the corollary.

\subsection{Proof of Lemma \ref{lem: HB}}
In Section \ref{sec: consistency a} it was shown that $\mathbb{M}_n(a)=\frac{\partial \ell_n(a)}{\partial a}$ satisfies, for positive constants $K_1$, $K_2$ and $K_3$,
\begin{align*}
\frac{\mathbb{M}_n(a)}{\log^2(n/a)} \left\{
\begin{array}{ll}
	\leq -K_1, & \mbox{for } a\geq\overline{a}_n\\
	\geq K_2\log (n/\underline{a}_n), & \mbox{for } a\in[\underline{a}_n^*,\underline{a}_n]\\
	\geq -K_3, & \mbox{for } a\leq\underline{a}_n^*,
\end{array}
\right.
\end{align*}
where $\underline{a}_n^*=\underline{a}_n\log n/(1+\log n)$.
Furthermore, the constant $K_2$ can be chosen arbitrarily large by choosing $B$ large enough, while the constant $K_3$ is fixed.

For $a\geq C\overline{a}_n$ with $C\geq 3$, we have
$$\ell_n(a)-\ell_n(2\overline{a}_n)\leq -K_1\log^2 (n/\overline{a}_n)(a-2\overline{a}_n)\leq -K_4\log^2 (n/\overline{a}_n)\overline{a}_n$$
with $K_4=K_1(C-2)$. Consequently $e^{\ell_n(a)}\leq e^{\ell_n(2\overline{a}_n)-K_4\log^2 (n/\overline{a}_n)\overline{a}_n}$ for $a\geq C\overline{a}_n$. Since also $e^{\ell_n(a)}\geq e^{\ell_n(2\overline{a}_n)}$ for $a\in[\overline{a}_n,2\overline{a}_n]$, we find
\begin{align}
\Pi(a\geq C\overline{a}_n|Y)\leq\frac{\int_{C\overline{a}_n}^{\infty}e^{\ell_n(a)}\pi(a)da}{\int_{\overline{a}_n}^{2\overline{a}_n}e^{\ell_n(a)}\pi(a)da}\leq\frac{\Pi\big([C\overline{a}_n,\infty)\big)e^{-K_4\log^2 (n/\overline{a}_n)\overline{a}_n}}{\Pi([\overline{a}_n,2\overline{a}_n])}.\label{eq: hb:ub}
\end{align}
Note that by Assumption \ref{assump: HB}
\begin{align*}
\Pi\big([\overline{a}_n,2\overline{a}_n]\big)\gtrsim \overline{a}_n^{1-c_3} e^{-c_2\overline{a}_n}\gg e^{-K_4 \log^2(n/\overline{a}_n)\overline{a}_n},
\end{align*}
hence the right hand side of \eqref{eq: hb:ub} tends to zero.

The analysis of the left tail goes similarly. Note that for $a<\underline{a}_n^*/2$ we have
$\ell_n(\underline{a}_n^*)-\ell_n(a)\geq -K_3 (\underline{a}_n^*-a)\log^2(n/\underline{a}_n),$ hence $e^{\ell_n(a)}\leq e^{\ell_n(\underline{a}_n^*)+K_3\underline{a}_n\log^2 (n/\overline{a}_n)}$ and analogously for $(\underline{a}_n+\underline{a}_n^*)/2<a<\underline{a}_n$ we have $\ell_n(a)-\ell_n(\underline{a}_n^*)\geq K_2 (a-\underline{a}_n^*)\log^3(n/\underline{a}_n)$, which implies $e^{\ell_n(a)}\geq e^{\ell_n(\underline{a}_n^*)+K_2(\underline{a}_n/4)\log^2(n/\underline{a}_n)}$. Therefore
\begin{align}
\Pi(a\leq \underline{a}_n^*|Y)\leq\frac{\int_1^{\underline{a}_n^*}e^{\ell_n(a)}\pi(a)da}{\int_{(\underline{a}_n+\underline{a}_n^*)/2}^{\underline{a}_n}e^{\ell_n(a)}\pi(a)da} \leq \frac{\Pi([1,\underline{a}_n])e^{K_3\underline{a}_n\log^2(n/\underline{a}_n)}}{\Pi([(\underline{a}_n+\underline{a}_n^*)/2,\underline{a}_n])e^{K_2(\underline{a}_n/4)\log^2(n/\underline{a}_n)}}.\label{eq: help:hb:lemma}
\end{align}
Since 
$$\Pi([(\underline{a}_n+\underline{a}_n^*)/2,\underline{a}_n])^{-1} \lesssim \log (n)\underline{a}_n^{c_5-1} e^{c_6\underline{a}_n}\ll e^{K_2(\underline{a}_n/8) \log^2(n/\underline{a}_n)},$$
for large enough choice of $K_2$, the right hand side of \eqref{eq: help:hb:lemma} tends to zero, finishing the proof of the lemma.

\section{Technical Lemmas}

\begin{lemma}\label{lem: bound term}~Let $i,m \in \mathbb{N}$ and $a\geq 1$, then for any $n/a\geq e^m$
\begin{align*}
\frac{ne^{i/a}i^m}{a^m(ae^{i/a}+n)^2}\leq\frac{1}{a}\log^m\big(\frac{n}{a}\big)\vee e\frac{a^{-m}}{n}.
\end{align*}
\end{lemma}

\begin{proof}~Assume first that $i\leq I_a\equiv a\log(n/a)$. Note that the function $f(x)=e^{x/a}(x/a)^m$ is monotone decreasing on $(-\infty,-ma]$ and monotone increasing on $[-ma,\infty]$. Then by the inequality $ae^{i/a}+n\geq n$,
\begin{align*}
\frac{ne^{i/a}i^m}{a^m(ae^{i/a}+n)^2}\leq\frac{e^{i/a}(i/a)^m}{ n} \leq\frac{1}{a}\log^m\big(\frac{n}{a}\big)\vee e\frac{a^{-m}}{n}.
\end{align*}
Next assume that $i>I_a$. Note that the derivative of the function $f(x)=e^{-x/a}x^m$ is $f'(x)=e^{-x/a}x^{m-1}(m-x/a)$, hence the function $f(i)$ is monotone decreasing for $i\geq am$. Thus for $n/a\geq e^{m}$, $f(i)$ takes its maximum at $i=I_a$, which implies that
\begin{align*}
\frac{ne^{i/a}i^m}{a^m(ae^{i/a}+n)^2}\leq \frac{ne^{-i/a}i^m}{a^{m+2}}  \leq\frac{1}{a}\log^m\big(\frac{n}{a}\big).
\end{align*}
\end{proof}

\begin{lemma}\label{lem: bound sum}~Let $l>r\geq 0$, then for $n/a\geq e^{l-r}$
\begin{align*}
\sum\limits_{i=1}^{\infty}\frac{e^{ir/a}}{(ae^{i/a}+n)^l}\lesssim\frac{n^{r-l}}{a^{r-1}}\log\big(\frac{n}{a}\big).
\end{align*}
\end{lemma}

\begin{proof}~First note that following from the inequality $ae^{i/a}+n\geq n$ and the sum of geometric series we get
\begin{align*}
\sum\limits_{i=1}^{I_a}\frac{e^{ir/a}}{(ae^{i/a}+n)^l}\leq n^{-1}\sum\limits_{i=1}^{I_a}e^{ir/a}\lesssim  \frac{n^{r-l}}{a^{r-1}}\log\big(\frac{n}{a}\big),
\end{align*}
where  $I_a\equiv a\log(n/a)$. Then similarly, using the inequality $ae^{i/a}+n\geq ae^{i/a}$ and  the sum of geometric series,
\begin{align*}
\sum\limits_{I_a}^\infty\frac{e^{ir/a}}{(ae^{i/a}+n)^l}\leq a^{-l}\sum\limits_{I_a}^\infty e^{(r-l)i/a}\leq \frac{n^{r-l}}{a^r}\frac{1}{e^{(l-r)/a}-1}\lesssim\frac{n^{r-l}}{a^{r-1}}\log\big(\frac{n}{a}\big)
\end{align*}
because $e^{(l-r)/a}-1\geq \frac{l-r}{a}$ and $\log\big(\frac{n}{a}\big)\geq l-r$ for $\frac{n}{a}\geq e^{l-r}$.
\end{proof}

\begin{lemma}[Lemma C.11 of \cite{vanderpas2017}]\label{Lem: Diff}
For any stochastic process $(V_{a}: a>0)$ with continuously differentiable sample paths $a\mapsto V_a$, 
with derivative written as $\dot V_a$,
\begin{align*}
E( V_{a_2}-V_{a_1})^2\leq (a_2-a_1)^2\sup_{a\in[a_1,a_2]}E \dot V_{a}^2.
\end{align*}
\end{lemma}

\bibliographystyle{acm}
\bibliography{bibi}

\begin{thebibliography}{10}

\bibitem{belitser:2014}
{\sc Belitser, E.}
\newblock On coverage and local radial rates of credible sets.
\newblock {\em Ann. Statist. 45}, 3 (06 2017), 1124--1151.

\bibitem{bhatt:2015}
{\sc Bhatt, S., Weiss, D., Cameron, E., Bisanzio, D., Mappin, B., Dalrymple,
  U., Battle, K., Moyes, C., Henry, A., Eckhoff, P., et~al.}
\newblock The effect of malaria control on plasmodium falciparum in africa
  between 2000 and 2015.
\newblock {\em Nature 526}, 7572 (2015), 207.

\bibitem{bhattacharya:pati}
{\sc Bhattacharya, A., and Pati, D.}
\newblock Adaptive {B}ayesian inference in the {G}aussian sequence model using
  exponential-variance priors.
\newblock {\em Statist. Prob. Letters 103\/} (2015), 100--104.

\bibitem{brown:1996}
{\sc Brown, L.~D., and Low, M.~G.}
\newblock Asymptotic equivalence of nonparametric regression and white noise.
\newblock {\em Ann. Statist. 24}, 6 (12 1996), 2384--2398.

\bibitem{cai:low:04}
{\sc Cai, T., and Low, M.}
\newblock An adaptation theory for nonparametric confidence intervals.
\newblock {\em Ann. Statist. 32\/} (2004), 1805--1840.

\bibitem{castillo:2014}
{\sc Castillo, I., Kerkyacharian, G., and Picard, D.}
\newblock Thomas bayes' walk on manifolds.
\newblock {\em Probability Theory and Related Fields 158}, 3 (Apr 2014),
  665--710.

\bibitem{castillo:nickl:2013}
{\sc Castillo, I., and Nickl, R.}
\newblock Nonparametric {B}ernsteinñ-von {M}ises theorems in gaussian white
  noise.
\newblock {\em Ann. Statist. 41}, 4 (08 2013), 1999--2028.

\bibitem{castillo:szabo:2018}
{\sc Castillo, I., and Szabo, B.}
\newblock Spike and slab empirical bayes sparse credible sets.
\newblock {\em arXiv preprint arXiv:1808.07721\/} (2018).

\bibitem{donoho1990}
{\sc Donoho, D.~L., Liu, R.~C., and MacGibbon, B.}
\newblock Minimax risk over hyperrectangles, and implications.
\newblock {\em Ann. Statist. 18}, 3 (09 1990), 1416--1437.

\bibitem{fmswwz:astonomy:16}
{\sc Faraway, J., Mahabal, A., and Berger, J.}
\newblock Robust {B}ayesian displays for standard inferences concerning a
  normal mean.
\newblock {\em Computational Statistics and Data Analysis 33}, 1 (2000),
  381--399.

\bibitem{ghosal:ghosh:vdv:00}
{\sc Ghosal, S., Ghosh, J.~K., and van~der Vaart, A.}
\newblock Convergence rates of posterior distributions.
\newblock {\em Ann. Statist. 28\/} (2000), 500--531.

\bibitem{ghosal:vdv:07}
{\sc Ghosal, S., and van~der Vaart, A.}
\newblock Convergence rates of posterior distributions for non iid
  observations.
\newblock {\em Ann. Statist. 35}, 1 (2007), 192--223.

\bibitem{gine:2010}
{\sc Giné, E., and Nickl, R.}
\newblock Confidence bands in density estimation.
\newblock {\em Ann. Statist. 38}, 2 (04 2010), 1122--1170.

\bibitem{gine:nickl:2016}
{\sc Gin{\'e}, E., and Nickl, R.}
\newblock {\em Mathematical foundations of infinite-dimensional statistical
  models}.
\newblock Cambridge series in statistical and probabilistic mathematics. 2016.

\bibitem{kalaitzis:lawrence:2011}
{\sc Kalaitzis, A., and Lawrence, N.}
\newblock A simple approach to ranking differentially expressed gene expression
  time courses through {G}aussian process regression.
\newblock {\em BMC 12}, 1 (2011), 180.

\bibitem{knapikSVZ2012}
{\sc Knapik, B.~T., Szab{\'o}, B.~T., van~der Vaart, A.~W., and van Zanten,
  J.~H.}
\newblock Bayes procedures for adaptive inference in inverse problems for the
  white noise model.
\newblock {\em Probability Theory and Related Fields 164}, 3 (Apr 2016),
  771--813.

\bibitem{koriyama:kobayashi:2015}
{\sc Koriyama, T., and Kobayashi, T.}
\newblock Prosody generation using frame-based gaussian process regression and
  classification for statistical parametric speech synthesis.
\newblock In {\em 2015 IEEE International Conference on Acoustics, Speech and
  Signal Processing (ICASSP)\/} (April 2015), pp.~4929--4933.

\bibitem{Low:97}
{\sc Low, M.}
\newblock On nonparametric confidence intervals.
\newblock {\em Ann. Statist. 25\/} (1997), 2547--2554.

\bibitem{nussbaum:1996}
{\sc Nussbaum, M.}
\newblock Asymptotic equivalence of density estimation and gaussian white
  noise.
\newblock {\em Ann. Statist. 24}, 6 (12 1996), 2399--2430.

\bibitem{rasmussen:williams:2006}
{\sc Rasmussen, C., and Williams, C.}
\newblock {\em Gaussian processes for machine learning}.
\newblock MIT Press, Boston, 2006.

\bibitem{rasmussen2004gaussian}
{\sc Rasmussen, C.~E.}
\newblock Gaussian processes in machine learning.
\newblock In {\em Advanced lectures on machine learning}. Springer, 2004,
  pp.~63--71.

\bibitem{ray2017}
{\sc Ray, K.}
\newblock Adaptive bernstein-von mises theorems in gaussian white noise.
\newblock {\em Ann. Statist. 45}, 6 (12 2017), 2511--2536.

\bibitem{robins:2006}
{\sc Robins, J., and van~der Vaart, A.}
\newblock Adaptive nonparametric confidence sets.
\newblock {\em Ann. Statist. 34}, 1 (02 2006), 229--253.

\bibitem{rousseau:szabo:16:main}
{\sc Rousseau, J., and Szabo, B.}
\newblock Asymptotic frequentist coverage properties of bayesian credible sets
  for sieve priors.
\newblock {\em arXiv preprint arXiv:1609.05067\/} (2016).

\bibitem{rousseau:szabo:2015:main}
{\sc {Rousseau}, J., and {Szabo}, B.}
\newblock Asymptotic behaviour of the empirical bayes posteriors associated to
  maximum marginal likelihood estimator.
\newblock {\em Ann. Statist. 45\/} (2017), 833--865.

\bibitem{sniekers:2015}
{\sc Sniekers, S., and van~der Vaart, A.}
\newblock Adaptive {B}ayesian credible sets in regression with a {G}aussian
  process prior.
\newblock {\em Electron. J. Stat. 9}, 2 (2015), 2475--2527.

\bibitem{szabo:etal:2015}
{\sc Szab\'o, B.~T., Vaart, A.~W., and van Zanten, J.~H.}
\newblock Honest bayesian confidence sets for the l2-norm.
\newblock {\em Journal of Statistical Planning and Inference 166\/} (2015), 36
  -- 51.
\newblock Special Issue on Bayesian Nonparametrics.

\bibitem{szabo:vdv:vzanten:13}
{\sc Szabo, B.~T., van~der Vaart, A.~W., and van Zanten, J.~H.}
\newblock {Frequentist coverage of adaptive nonparametric Bayesian credible
  sets}.
\newblock {\em Annals of Statistics 43}, 4 (2015), 1391--1428.

\bibitem{tsybakov:2009}
{\sc Tsybakov, A.~B.}
\newblock Introduction to nonparametric estimation. revised and extended from
  the 2004 french original. translated by vladimir zaiats, 2009.

\bibitem{vanderpas2017}
{\sc van~der Pas, S., Szabo, B., and van~der Vaart, A.}
\newblock Adaptive posterior contraction rates for the horseshoe.
\newblock {\em Electron. J. Statist. 11}, 2 (2017), 3196--3225.

\bibitem{pas:etal::2016}
{\sc van~der Pas, S., Szabo, B., and van~der Vaart, A.}
\newblock Uncertainty quantification for the horseshoe (with discussion).
\newblock {\em Bayesian Anal. 12}, 4 (12 2017), 1221--1274.

\bibitem{vzanten:vdv:07}
{\sc van~der Vaart, A., and van Zanten, J.~H.}
\newblock Bayesian inference with rescaled {G}aussian process priors.
\newblock {\em Electron. J. Statist. 1\/} (2007), 433--448.

\bibitem{vandervaart2009}
{\sc van~der Vaart, A.~W., and van Zanten, J.~H.}
\newblock Adaptive bayesian estimation using a gaussian random field with
  inverse gamma bandwidth.
\newblock {\em Ann. Statist. 37}, 5B (10 2009), 2655--2675.

\bibitem{vdv:wellner:96}
{\sc Van Der~Vaart, A.~W., and Wellner, J.~A.}
\newblock Weak convergence.
\newblock In {\em Weak convergence and empirical processes}. Springer, 1996,
  pp.~16--28.

\end{thebibliography}

\end{document}